\documentclass[11pt]{article}
\usepackage{graphicx}
\usepackage{amsmath}
\usepackage{amssymb}
\usepackage{theorem}
\usepackage[pdftex]{color}
\usepackage{soul}
\usepackage{enumerate}
\usepackage{todonotes}

   \usepackage{hyperref}

\numberwithin{equation}{section}

\newtheorem{thm}{Theorem}[section]
\newtheorem{ltheorem}{Theorem} 
\newtheorem{lemma}[thm]{Lemma}
\newtheorem{prop}[thm]{Proposition}
\newtheorem{cor}[thm]{Corollary}
{\theorembodyfont{\rmfamily}
\newtheorem{defn}[thm]{Definition}

\newtheorem{rmk}[thm]{Remark}
}

\DeclareFontFamily{OML}{rsfs}{\skewchar\font'177}
\DeclareFontShape{OML}{rsfs}{m}{n}{ <5> <6> rsfs5 <7> <8> <9> rsfs7
  <10> <10.95> <12> <14.4> <17.28> <20.74> <24.88> rsfs10 }{}
\DeclareMathAlphabet{\mathfs}{OML}{rsfs}{m}{n}

\newcommand{\qed}{\hfill \mbox{\raggedright \rule{.07in}{.1in}}}
 
\newenvironment{proof}{\vspace{1ex}\noindent{\bf
Proof}\hspace{0.5em}}{\hfill\qed\vspace{1ex}}
\newenvironment{pfof}[1]{\vspace{1ex}\noindent{\bf Proof of
#1}\hspace{0.5em}}{\hfill\qed\vspace{1ex}}

\newcommand{\R}{{\mathbb R}}

\newcommand{\Z}{{\mathbb Z}}

\newcommand{\bbS}{{\mathbb S}}

\newcommand{\cC}{{\mathcal C}}
\newcommand{\cN}{{\mathcal N}}
\newcommand{\cR}{{\mathcal R}}

\newcommand{\mfC}{{\mathfs C}}
\newcommand{\mfD}{{\mathfs D}}
\newcommand{\mfG}{{\mathfs G}}
\newcommand{\mfS}{{\mathfs S}}

\newcommand{\tS}{{\widetilde S}}

\newcommand{\hj}{{\hat j}}
\newcommand{\hE}{{\widehat E}}
\newcommand{\hW}{{\widehat W}}

\newcommand\x{{\bf x}}

\newcommand{\eps}{\epsilon}
\newcommand{\ve}{\varepsilon}
\newcommand{\Deg}{\operatorname{Deg}}
\newcommand{\Reg}{\operatorname{Reg}}

\newcommand{\Lip}{\operatorname{Lip}}
\newcommand{\dist}{\operatorname{dist}}
\newcommand{\Int}{\operatorname{int}}

\newcommand{\SMALL}{\textstyle}
\title{Superdiffusive central limit theorems for \\ geodesic flows on nonpositively curved surfaces}

\author{Yuri Lima
\thanks{Instituto de Matemática e Estatística, Universidade de São Paulo, Rua do Matão, 1010, Cidade Universitária, 05508-090, São Paulo -- SP, Brazil.
Email: yurilima@gmail.com}
\and Carlos Matheus
\thanks{CNRS \& \'Ecole Polytechnique, CNRS (UMR 7640), 91128, Palaiseau, France.
Email: carlos.matheus@math.cnrs.fr}
 \and Ian Melbourne
  \thanks{Mathematics Institute, University of Warwick, Coventry CV4 7AL, United Kingdom.
 Email: i.melbourne@warwick.ac.uk
}}

\date{11 December 2025}

\begin{document}

\maketitle

\begin{abstract}
We prove a nonstandard central limit theorem and weak invariance principle,
with superdiffusive normalisation $(t\log t)^{1/2}$, for geodesic flows on a class of nonpositively curved
surfaces with flat cylinder. We also prove that correlations decay at rate~$t^{-1}$.
An important ingredient of the proof, which is of independent interest, is an improved
results on the regularity of the stable/unstable foliations induced by the Green bundles.
\end{abstract}

\tableofcontents

 \section{Introduction} 
 \label{sec-intro}

It is well-known that geodesic flows on negatively curved closed manifolds have very strong statistical properties.
Ergodicity with respect to volume was proved for surfaces in~\cite{Hopf39} and for general dimension in~\cite{Anosov67}. The Bernoulli property was shown in~\cite{OrnsteinWeiss73,Ratner74}. Statistical limit laws such as the central limit theorem (CLT), the weak invariance principle (WIP), also known as the functional CLT, and the almost sure invariance principle quickly followed~\cite{DenkerPhilipp84,Ratner73}.
In major breakthroughs, exponential decay of correlations was established in~\cite{Dolgopyat98} for surfaces and~\cite{Liverani04} in general dimension.

In contrast, there are few results on statistical properties beyond the Bernoulli property for 
geodesic flows on nonpositively curved manifolds when there exist points of zero curvature.
In~\cite{LMM24}, we initiated the systematic study of such properties for geodesic flows.
In particular, we considered a class of closed surfaces with one flat geodesic
and proved polynomial decay of correlations and various statistical limit laws including the CLT and WIP
(with standard normalisation).
In this paper, we provide examples of geodesic flows on nonpositively curved surfaces
that satisfy a CLT and WIP with \emph{nonstandard normalisation}.

Our approach to studying geodesic flows on nonpositively curved manifolds resembles that for (semi)dispersing billiards~\cite{ChernovMarkarian}.
General results on statistical limit laws and decay of correlations seem infeasible,
but it is possible to analyse interesting classes of examples.
Moreover, whereas rigorous results for dispersing billiards have been restricted to planar billiards due to the complicated structure of singularities under iteration, there is in principle no such restriction for geodesic flows since they are smooth.
On the other hand, the prerequisites for the study of dispersing billiards are firmly established. 
For geodesic flows on nonpositively curved manifolds, the corresponding prerequisites are still under development, so for the moment we focus on surfaces too.
A case in point is the smoothness of the foliations induced by the Green bundles and the necessity
to extend existing results such as those in~\cite{GerberNitica99,GerberWilkinson99}, see Section~\ref{sec:GW}.

\vspace{1ex}
The class of geodesic flows that we study in this paper is as follows.
Let $r\in[5,\infty)$ and fix $0<L<\ve_0$.
Define $\cR$ to be the surface of revolution
with profile~$\xi$ given by
$\xi(s)=1$ for $|s|\leq L$ and $\xi(s)=1+(|s|-L)^r$ for $L\leq |s|\leq \ve_0$.
We form 
a closed $C^4$ Riemannian surface $S$ of nonpositive curvature by isometrically gluing two negatively curved surfaces with boundary to the circles $\beta_\pm$ bounding $\cR$,
see Figure~\ref{fig:surface-intro}.
We call $S$ a \emph{surface with flat cylinder $\cC=[-L,L]\times\bbS^1$}, where
$\bbS^1=\R/(2\pi\Z)$.
Away from~$\cC$ the surface $S$ has negative curvature.

\begin{figure}[hbt!]
\centering
\def\svgwidth{11cm}
   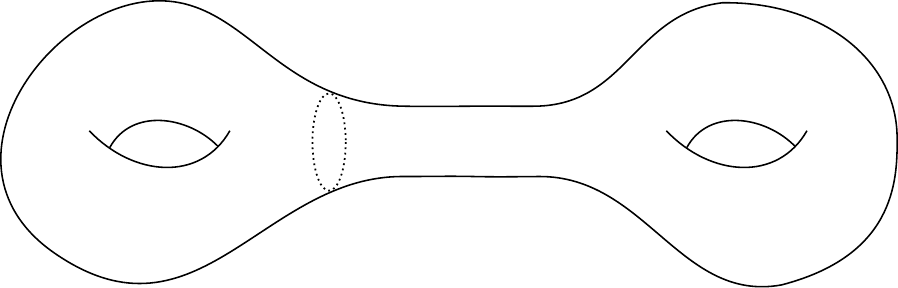
\caption{Surface
with flat cylinder $\cC$.
The region $\cR$ between the two curves $\beta_\pm$ is a surface of revolution
with profile $\xi$.}
\label{fig:surface-intro}
\end{figure}

Let $g_t:M\to M$ denote the geodesic flow on the unit tangent bundle
$M=T^1S$, and let $\mu$ denote the normalised Riemannian volume on the 
three-dimensional manifold $M$. Then $\mu$ is an ergodic $g_t$-invariant
probability measure \cite{Pesin77b}. 

The flat cylinder $\cC$ is foliated by two families of periodic orbits with clockwise/counterclockwise orientation.
Let $v:M\to\R$ be H\"older with $\int_M v\,d\mu=0$. We define $I_v=\alpha_0^2+\alpha_\pi^2$ where
$\alpha_0$ and $\alpha_\pi$ are the averages of $v$ over these two families of periodic orbits.
(The notation derives from the explicit formulas~\eqref{eq:alpha} given in Section~\ref{sec:WIP}.)
For $t>0$, define $v_t=\int_0^t v\circ g_s\,ds:M\to\R$.
Our first main result is a CLT for $v_t$ with standard normalisation $t^{1/2}$ or nonstandard normalisation $(t\log t)^{1/2}$ depending on whether $I_v=0$ or $I_v>0$.

\begin{ltheorem}[Nonstandard/Standard CLT] \label{thm:CLT}
If $I_v>0$, then $v$ satisfies a nonstandard CLT. That is,
$(t\log t)^{-1/2}v_t \to_d N(0,\sigma_v^2)$ as $t\to\infty$,
where $\sigma_v^2=b_SI_v>0$ and $b_S$ is a positive constant depending only on the surface $S$.

If $I_v=0$, then there exists $\sigma^2\ge0$, typically nonzero, such that
$v$ satisfies a standard CLT. That is,
$t^{-1/2}v_t \to_d N(0,\sigma^2)$ as $t\to\infty$.
\end{ltheorem}

\begin{rmk} As in~\cite[Remark~4.6]{LMMsub}, the word ``typically'' is interpreted here in the very strong sense that $\sigma^2=0$ only within a closed subspace of infinite codimension amongst H\"older observables $v$ with $\int_M v\,d\mu=I_v=0$. See for example the discussion in~\cite[End of Section~4]{HM07}.

We note that the vanishing of $\int_Mv\,d\mu$ is ``without loss of generality'' since one can always centre the observable $v$ by considering $v-\int_Mv\,d\mu$.
The vanishing of $I_v$ is ``codimension two'' since $\alpha_0=\alpha_\pi=0$
may occur in generic two-parameter families. In contrast, the degenerate case $\sigma^2=0$ 
requires infinitely many coincidences beyond the constraint $I_v=0$.
\end{rmk}

In the case $I_v>0$, we also prove a nonstandard WIP (which implies the CLT).
Define the sequence of continuous functions 
\[
W_n\in C[0,1], \qquad W_n(t)=(n\log n)^{-1/2}v_{nt},\quad n\ge1.
\]
Since each function $W_n$ depends on the initial condition in the probability space $(M,\mu)$, we
can regard $W_n$ as a random element in the metric space $(C[0,1],\|\;\|_\infty)$.

\begin{ltheorem}[Nonstandard WIP] \label{thm:WIP}
If $I_v>0$, then $v$ satisfies a nonstandard WIP. That is,
$W_n \to_w W$ in $C[0,1]$ as $n\to\infty$,
where $W$ is a Brownian motion with variance $\sigma_v^2$. 
\end{ltheorem}

There are only a few situations in deterministic dynamical systems where 
the nonstandard CLT/WIP is known to hold.
The first example was for intermittent maps $f$ satisfying $f(x)=x+bx^{3/2}$ at a neutral fixed point $x=0$; see~\cite{Gouezel04} for the nonstandard CLT and~\cite{DedeckerMerlevede09} for the nonstandard WIP.
For various billiard examples, namely the infinite horizon Lorentz gas, Bunimovich stadium and billiards with cusps, we refer to~\cite{BalintChernovDolgopyat11,BalintGouezel06,LMMsub,SzaszVarju07}.

Our proofs of Theorems~\ref{thm:CLT} and~\ref{thm:WIP} closely follows abstract arguments in~\cite{BalintChernovDolgopyat11} and~\cite{LMMsub}, using the latter to reduce to a nonstandard CLT for a simpler class of observables, and using the former to establish this simplified nonstandard CLT.

Our final main result concerns the rate of decay of correlations for the flow $g_t$.

\begin{ltheorem} \label{thm:decay}
 Let $v,\,w:M\to\R$ be sufficiently smooth observables. Then there is a constant
$C>0$ such that 
$$
\left|\int_M v \cdot (w\circ g_t)\,d\mu-\int_Mv\,d\mu\int_Mw\,d\mu\right|
\leq C t^{-1}
\quad\text{for all $t>0$}.
$$
\end{ltheorem}

\begin{rmk} An observable is ``sufficiently smooth'' if it is $C^k$ in the flow direction for some $k\ge0$ that depends only on the surface $S$. In 
particular, if we choose the surface $S$ to be $C^\infty$ away from $\cC$, then any observable that is $C^\infty$ on $M$ and flat near $\{s=\pm L\}$ is sufficiently smooth.

The analogous result for the infinite horizon Lorentz gas and Bunimovich stadia  was obtained in~\cite{BBM19}.
\end{rmk}

\begin{rmk} \label{rmk:lower}
 The upper bound in Theorem~\ref{thm:decay} is sharp in the sense that if $v$ is H\"older with $\int_Mv\,d\mu=0$ and $I_v\neq0$,  then 
\[
t\int_M v \cdot (v\circ g_t)\,d\mu\not\to0
\quad\text{as $t\to\infty$}.
\]
This is a consequence of the nonstandard limit law in Theorem~\ref{thm:CLT}.
See~\cite[Corollary~1.3]{BalintGouezel06} or~\cite[Proposition~9.14]{BBM19}.
\end{rmk}

The remainder of this paper is organised as follows.
In Section \ref{sec:surfaces}, we review known facts about the geometry and dynamics of 
geodesic flows on nonpositively curved surfaces, with special attention to surfaces with flat cylinder.
In Section \ref{sec:R-dynamics}, we make a systematic study of the dynamics of the geodesic flow on the surface of revolution $\cR$.
In Section~\ref{sec:GW}, we prove sufficient regularity of the stable/unstable foliations
induced by the Green bundles, partly extending~\cite{GerberWilkinson99}.
In Section~\ref{sec:f}, we construct a uniformly hyperbolic first return map $f:\Sigma_0\to\Sigma_0$ satisfying the Chernov axioms~\cite{Chernov-1999}, showing that $f$ is modelled by a Young tower with exponential tails~\cite{Young98}.

In Section~\ref{sec:CLT}, we recall and slightly extend an abstract result on the nonstandard CLT due to~\cite{BalintChernovDolgopyat11}.
In Section~\ref{sec:WIP}, we prove Theorems~\ref{thm:CLT} and~\ref{thm:WIP}.
Finally, in Section~\ref{sec:decay} we prove Theorem~\ref{thm:decay} as well as a related result, Theorem~\ref{thm:decaymap}, for a suitable global Poincar\'e map $g$.
The Chernov axioms are stated for convenience in Appendix~\ref{app:C}.

\vspace{-2ex}
\paragraph{Notation}
We use ``big O'' and $\ll$ notation interchangeably, writing $a_n=O(b_n)$ or $a_n\ll b_n$
if there are constants $C>0$, $n_0\ge1$ such that
$a_n\le Cb_n$ for all $n\ge n_0$.
We write $a_n\approx b_n$ if $a_n\ll b_n$ and $b_n\ll a_n$, and 
$a_n\sim b_n$ if $a_n/b_n\to1$.

\section{Nonpositively curved surfaces}\label{sec:surfaces}
In this section, we recall some known facts from differential geometry, especially
for geodesic flows on nonpositively curved surfaces and for surfaces of revolution.
Standard references are \cite{Ballmann,Eberlein}. 
We also give a precise description
of the class of surfaces considered in this article, and then describe their properties
that will be used in the sequel.

\subsection{Geodesic flows}\label{sec:geodesic-flows}

Let $S$ be a $C^4$ closed Riemannian surface. Let $M=T^1S$ be its unit tangent bundle,
which is a closed three dimensional Riemannian manifold. 
The \emph{geodesic flow} on $S$ is the flow $g_t:M\to M$ defined by
$g_t(x)=\gamma'_x(t)$, where $\gamma_x:\R\to S$ is the unique geodesic such that $\gamma'_x(0)=x$.
The volume form on $S$ induces a smooth $g_t$-invariant probability measure $\mu$ on $M$.

For $p\in S$, let $K(p)$ be the Riemannian curvature at $p$.
We assume that $S$ has nonpositive Riemannian curvature: $K(p)\leq 0$ for all $p\in S$. Then we have a partition
$M=\Deg\cup\Reg$ 
into the \emph{degenerate set}\footnote{The classical literature uses the terminology
``singular set'' instead of ``degenerate set'', but here we reserve the term
``singular'' for the dynamical setting of the Chernov axioms.} $\Deg$ and the \emph{regular set} $\Reg$ 
defined by
\begin{align*}
\Deg & =\{x\in M:K(\gamma_x(t))=0\text{ for all }t\in\R\},
\\ 
\Reg & =\{x\in M:K(\gamma_x(t))<0\text{ for some }t\in\R\}.
\end{align*}

The dynamical properties of $g_t$ are usually studied via Jacobi fields.
We refer the reader to \cite[Sect.\ 2.1]{LMM24} for their definition and an
exposition of their basic properties, including how they define invariant subspaces
and foliations. Below, we just recall some of their properties.

Using Jacobi fields that remain bounded in positive/negative time, we can define
$dg_t$-invariant one-dimensional subspaces $\hE^{s/u}_x\subset T_xM$ for
each $x\in M$.\footnote{We reserve the notation $E^{s/u}_x$ for the stronger
notion of stable/unstable
subspace in the sense of hyperbolic dynamics, as described in Section
\ref{sec:f}.}
These are the {\em Green bundles}, introduced in \cite{Green58}.
As in \cite{LMM24}, we refer to $\hE^{s/u}_x$ as the stable/unstable
subspaces of $x\in M$. These subspaces are orthogonal to  the one-dimensional 
subspace $Z_x$ tangent to the geodesic flow.
Moreover, $\hE^s_x=\hE^u_x$ if and only if $x\in \Deg$; hence
 $\hE^s_x\oplus Z_x \oplus \hE^u_x=T_xM$ if and only if $x\in\Reg$. 

Using Busemann functions and horospheres \cite[Sect.\ 2.1]{LMM24},
for each $x\in M$ we define $g_t$-invariant $C^1$ curves $\hW^{s/u}_x$
tangent to $\hE^{s/u}_x$, called the \emph{stable/unstable manifolds} for the flow.

There is a natural metric on $M$,
called the \emph{Sasaki metric}, which is the product of horizontal and vertical vectors,
see e.g.\ \cite[Chapter 3, Exercise 2]{doCarmo}. 
For $\delta>0$ small, there is an equivalent version of the Sasaki metric \cite[\S 17.6]{Katok-Hasselblatt-Book},
which we call the \emph{$\delta$-Sasaki metric}.
In our calculations, we will fix a $\delta$-Sasaki
metric for $\delta$ small enough and denote it simply by $\|\cdot\|$.

\subsection{Surfaces of revolution}\label{sec:revolution}

Let $I\subset\R$ be a compact interval, and let $\xi:I\to\R$ be a positive $C^4$ function.
The \emph{surface of revolution}
defined by $\xi$ around the $x$ axis is the surface $\cR$ with global chart
$\Xi:I\times\bbS^1\to \R^3$ given by
$\Xi(s,\theta)=(s,\xi(s)\cos\theta,\xi(s)\sin\theta)$.
We collect some known facts about these surfaces, referring to \cite{doCarmo76} for the details.

By \cite[Example 4, p.\ 161]{doCarmo76},
the curvature at $p=\Xi(s,\theta)$ is equal to
\begin{equation} \label{eq:K}
K(p)=-\frac{\xi''(s)}{\xi(s)[1+(\xi'(s))^2]^2}\cdot
\end{equation}

Geodesics on surfaces of revolution 
satisfy the so-called \emph{Clairaut relation}.
Let $\gamma(t)=\Xi(s(t),\theta(t))$ be a geodesic, and let $\psi(t)\in\bbS^1$ be the angle that
the circle $s=s(t)$, more precisely its image under $\Xi$, makes with $\gamma$ at $\gamma(t)$.
The Clairaut relation states that the value
\begin{equation} \label{eq:Clairaut}
c=\xi(s(t))\cos\psi(t)=\xi(s(t))^2 \theta'(t)
\end{equation}
is constant along $\gamma$.
We call $c$ the \emph{Clairaut constant} of $\gamma$ and of all of its tangent vectors.

Let $\gamma(t)=\Xi(s(t),\theta(t))$ be a geodesic with Clairaut constant $c$.
Then $s(t)$ satisfies
the \emph{equation of geodesics} 
 \cite[Example 5,  p.\ 255]{doCarmo76},
\[
[1+\xi'(s)^2](s')^2+\frac{c^2}{\xi(s)^2}=1.
\]

For a fixed $s_0\in\R$, the curve $\Xi(s_0,\theta)$ is called a \emph{meridian}.
We fix an orientation of $\bbS^1$ and induce it on the meridians.
Observe that $\cR$ is diffeomorphic to $I\times\bbS^1$ and $T^1\cR$ is
diffeomorphic to $\cR\times \bbS^1\cong I\times\bbS^1\times\bbS^1$,
where $(p,\psi)\in \cR\times \bbS^1$ is identified to the unit tangent vector with basepoint 
$p$ that makes
an angle $\psi$ with the oriented meridian passing though $p$.
Writing $p=\Xi(s,\theta)$, we refer to the coordinates 
$(s,\theta,\psi)\in I\times\bbS^1\times \bbS^1$ as
\emph{Clairaut coordinates} on $T^1\cR$.

Using the Clairaut relation \eqref{eq:Clairaut}, we define
the \emph{Clairaut function} $c:T^1\cR\to\R$ by
$c(s,\theta,\psi)=\xi(s)\cos\psi$.

The \emph{Clairaut metric}
on $T^1\cR$ is the Riemannian metric given by the canonical product on
$I\times\bbS^1\times \bbS^1$; it is equivalent to the $\delta$-Sasaki metric $\|\cdot\|$.

\subsection{Surfaces with flat cylinder}\label{sec:our-surface}

We now define the class of surfaces that we are interested in in this paper.
They exhibit two special features:
the only region of zero curvature is a flat cylinder,
and on a neighbourhood of this cylinder the surface is a particular surface of revolution.

We begin by defining the surface of revolution.
Fix $r>4$ and $0<L<\ve_0$.
Consider the profile function
$\xi:[-\ve_0,\ve_0]\to\R$ given by
\[
\xi(s)=
\begin{cases}
1, & |s|\leq L\\
1+(|s|-L)^r, & |s|\geq L.
\end{cases}
\]
This defines a surface of revolution $\cR$ with boundary meridians
$\beta_\pm=\Xi(\pm\ve_0,\theta)$.
We call $\cC=[-L,L]\times\bbS^1$ the \emph{flat cylinder} and
$\cN=\cR\setminus\cC$ the \emph{neck}.

\begin{defn}
A $C^4$ surface of nonpositive curvature $S$
is a \emph{surface with flat cylinder} if $S$ contains a copy of $\cR$ and 
$K|_{S\setminus\cC}<0$.
See Figure~\ref{fig:surface}.
\end{defn}

\begin{figure}[hbt!]
\centering
\def\svgwidth{11cm}
   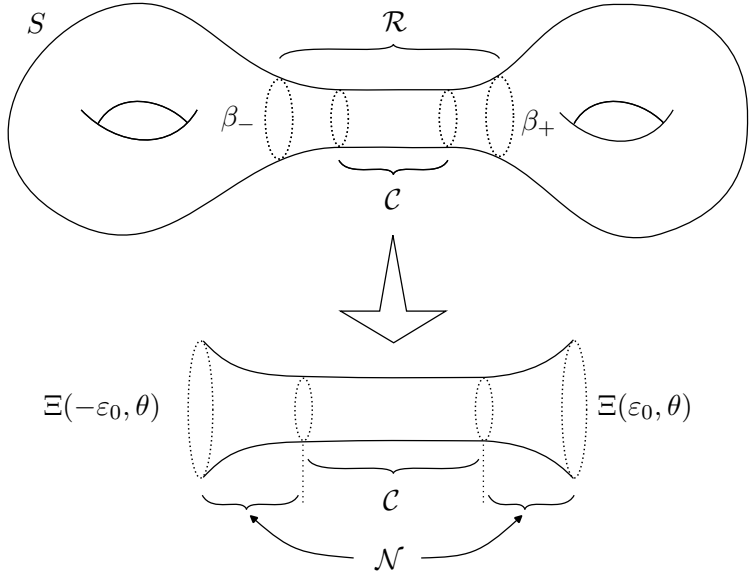
\caption{An example of a surface $S$ with flat cylinder.}\label{fig:surface}
\end{figure}

Such surfaces indeed exist, and can be obtained by interpolating the neck with
a hyperbolic surface ($K\equiv -1$) with one cusp on each side. Since near the cusp a hyperbolic
surface is a surface of revolution, it is enough to interpolate its profile function with the function $\xi$,
in a way that the resulting function is strictly convex for $s>L$. This can be done as in
\cite[Appendix A.2]{Donnay}, using a partition of unity.

In the sequel, we fix a surface $S$ with flat cylinder and $r>4$.
Eventually, we specialise to $r\ge5$ as in the introduction since this is required in Theorem~\ref{thm:GW1} below. All other arguments hold also for $r\in(4,5)$.

In the Clairaut coordinates, $\Deg$ is the union of
$[-L,L]\times \bbS^1\times\{0\}$ and
$[-L,L]\times \bbS^1\times\{\pi\}$.
In particular, $\mu(\Deg)=0$, 
so \cite{Pesin77b} implies that the flow $g_t$ is ergodic. In fact, $g_t$ is Bernoulli;
see \cite{Pesin77a} and
\cite[Thm.\ 12.2.13]{Barreira-Pesin-Non-Uniform-Hyperbolicity-Book} for the classical proofs,
and \cite{Ledrappier-Lima-Sarig} for a proof using symbolic dynamics.

\subsection{Asymptotic, bouncing and crossing geodesics/vectors}
\label{sec:asymptotic}

As in~\cite{LMM24}, we
use the Clairaut function to distinguish some vectors in $T^1\cN$ that will play a key role
in the next sections. 
Let $x=(s,\theta,\psi)\in T^1\cN$ such that 
the geodesic starting at $x$ points towards $\cC$.
The vector $x$, and the corresponding geodesic with initial condition $x$,  is called:
\begin{description}
\item[Asymptotic] if $c(x)=\pm 1$: the geodesic path $g_{[0,\infty)}(x)$ is
asymptotic to $\cC$.
\item[Bouncing] if $|c(x)|>1$: there is $t>0$ such that $\psi(t)=0$ or $\pi$,
i.e.\ the geodesic path $g_{[0,t]}(x)$ spirals towards $\cC$, $g_t(x)$
is tangent to a meridian, and after that the geodesic path spirals away from $\cC$ in the reverse direction.
In such cases, the geodesic does not reach~$\cC$.
\item[Crossing] if $|c(x)|<1$: there is $t>0$ such that $s(t)=0$,
i.e.\ the geodesic path $g_t(x)$ spirals into $\cC$, crosses $\cC$ from one side to the other,
and then spirals away from~$\cC$.
\end{description}
Figure~\ref{fig:geodesics} depicts the three possibilities.

\begin{figure}[hbt!]
\centering
\def\svgwidth{15cm}
\begingroup%
  \makeatletter%
  \providecommand\color[2][]{%
    \errmessage{(Inkscape) Color is used for the text in Inkscape, but the package 'color.sty' is not loaded}%
    \renewcommand\color[2][]{}%
  }%
  \providecommand\transparent[1]{%
    \errmessage{(Inkscape) Transparency is used (non-zero) for the text in Inkscape, but the package 'transparent.sty' is not loaded}%
    \renewcommand\transparent[1]{}%
  }%
  \providecommand\rotatebox[2]{#2}%
  \newcommand*\fsize{\dimexpr\f@size pt\relax}%
  \newcommand*\lineheight[1]{\fontsize{\fsize}{#1\fsize}\selectfont}%
  \ifx\svgwidth\undefined%
    \setlength{\unitlength}{1189.23887691bp}%
    \ifx\svgscale\undefined%
      \relax%
    \else%
      \setlength{\unitlength}{\unitlength * \real{\svgscale}}%
    \fi%
  \else%
    \setlength{\unitlength}{\svgwidth}%
  \fi%
  \global\let\svgwidth\undefined%
  \global\let\svgscale\undefined%
  \makeatother%
  \begin{picture}(1,0.27294678)%
    \lineheight{1}%
    \setlength\tabcolsep{0pt}%
    \put(0,0){\includegraphics[width=\unitlength,page=1]{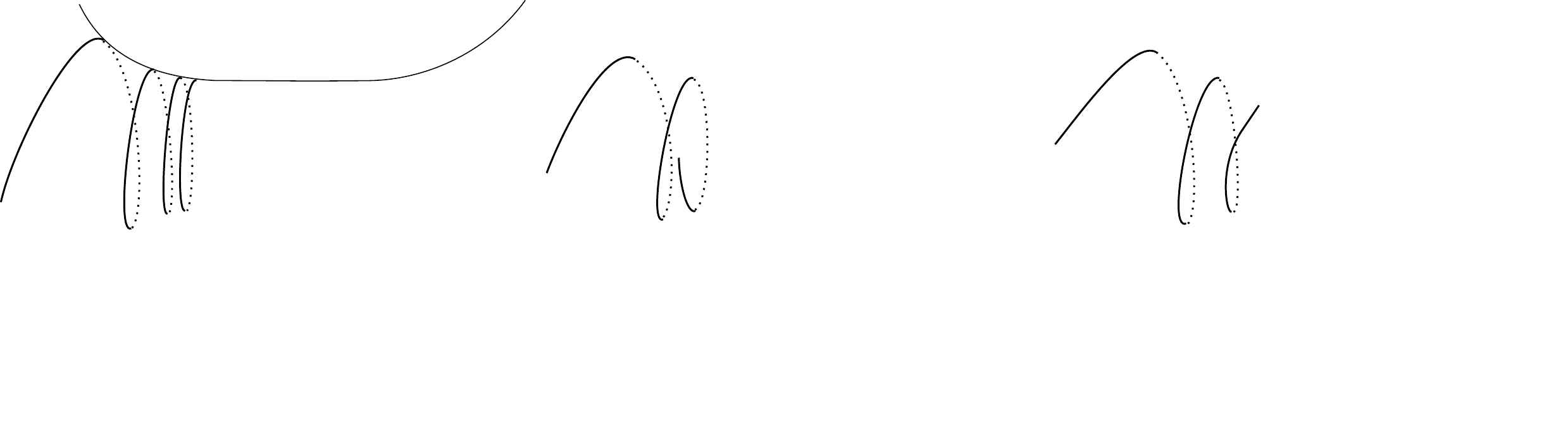}}%
    \put(0.15918733,0.00507479){\color[rgb]{0,0,0}\makebox(0,0)[lt]{\lineheight{1.25}\smash{\begin{tabular}[t]{l}(a)\end{tabular}}}}%
    \put(0.48962026,0.00507479){\color[rgb]{0,0,0}\makebox(0,0)[lt]{\lineheight{1.25}\smash{\begin{tabular}[t]{l}(b)\end{tabular}}}}%
    \put(0.82584358,0.00507479){\color[rgb]{0,0,0}\makebox(0,0)[lt]{\lineheight{1.25}\smash{\begin{tabular}[t]{l}(c)\end{tabular}}}}%
    \put(0,0){\includegraphics[width=\unitlength,page=2]{geodesics.pdf}}%
  \end{picture}%
\endgroup%

\caption{Geodesics that are (a) asymptotic, (b) bouncing, (c) crossing.}\label{fig:geodesics}
\end{figure}

\section{Dynamics of transitions in $\cR$}\label{sec:R-dynamics}

In this section, we initiate the study of the geodesic flow on a surface with flat cylinder as defined in Section~\ref{sec:our-surface}.
In particular, we describe in detail the transitions
in the surface of revolution $\cR$.

In Subsection~\ref{ss:Omega}, we construct a two-dimensional section $\Omega$, the \emph{transition section}, situated at (or near) the ends of $\cR$.
 We also recall an explicit formula from~\cite{LMM24} for 
the associated \emph{transition map} $f_0$.
In Subsection~\ref{ss:times}, we estimate transition times and derivatives of $f_0$.
The calculations in this section use the Clairaut coordinates
$x=(s,\theta,\psi)$ on $T^1\cR\cong[-\ve_0,\ve_0]\times\bbS^1\times\bbS^1$.

\subsection{The transition map $f_0=g_{2\Upsilon_0}:\Omega_{\rm in}\to\Omega_{\rm out}$}
\label{ss:Omega}

The transition section $\Omega$ we construct\footnote{The construction is the same as in~\cite{LMM24}  but the notation is different. The sections $\Omega_\pm$ there are called $\Omega_{\rm in/out}$ here (the $\pm$ is not related to stable/unstable or to $\pm\ve_0$, so was perhaps overused). The section $\Omega_0$ in~\cite{LMM24} is not part of the definition of $f_0$ and hence is no longer part of $\Omega$. It reappears in Section~\ref{sec:g}.

In addition, $\Omega$ is situated inside the neck at $\pm\ve_1$ instead of at $\pm\ve_0$ in~\cite{LMM24}. This ensures that the perturbed cross-sections $\Sigma_{\rm in/out}$ in Section~\ref{sec:f} lie inside $\cR$ so that the Clairaut function $c$ is defined.}
 allows a very simple description of the transitions
of geodesics in $\cR$.
Let $\ve_1=(L+\ve_0)/2$.
As detailed below, we define $\Omega=\Omega_{\rm in}\cup\Omega_{\rm out}$ where:
\begin{enumerate}[$\bullet$]
\item $\Omega_{\rm in}\subset\{\pm\ve_1\}\times\bbS^1\times\bbS^1$ is a neighbourhood of four families of geodesics entering $\cR$ and asymptotic to $\cC$;
\item $\Omega_{\rm out} \subset\{\pm\ve_1\}\times\bbS^1\times\bbS^1$
is a neighbourhood of four families of geodesics exiting $\cR$ and asymptotic to $\cC$ in backwards time;
\item Each of $\Omega_{\rm in}$ and $\Omega_{\rm out}$ is a disjoint union of four annular regions (diffeomorphic to $\bbS^1\times(-1,1)$).
\end{enumerate}

To construct $\Omega_{\rm in/out}$, we use the Clairaut function $c$ from Section~\ref{sec:revolution}.
There is a unique $\psi_0\in(0,\frac{\pi}{2})$ such that $\xi(\pm\ve_1)\cos\psi_0=1$.
Recalling the notation of asymptotic vector introduced in
Section~\ref{sec:asymptotic},
the vectors $(\pm\ve_1,\theta,\pm\psi_0)$, $\theta\in\bbS^1$, constitute four families of asymptotic
vectors with Clairaut constant $c=1$ and asymptotic to $\cC$.
Similarly, $(\pm\ve_1,\theta,\pm(\pi-\psi_0))$, $\theta\in\bbS^1$, constitute four families of asymptotic
vectors with $c=-1$ and asymptotic to $\cC$.
Of these, $(-\ve_1,\theta,\psi_0)$, $(\ve_1,\theta,-\psi_0)$,
$(-\ve_1,\theta,\pi-\psi_0)$, $(\ve_1,\theta,-(\pi-\psi_0))$ correspond to the four families of vectors that enter
$\cR$ and are asymptotic to $\cC$ as $t\to\infty$.
The remaining four families of vectors exit $\cR$ and are asymptotic to $\cC$ as $t\to-\infty$.

Focusing momentarily on $x=(-\ve_1,\theta,\psi_0)$, we define
$$
\Omega_1=\{x=(-\ve_1,\theta,\psi)\in T^1\cR:|c(x)-1|<\chi\}.
$$
Shrinking $\chi>0$, we can ensure that
$\Omega_1= \{-\ve_1\}\times \bbS^1\times I$
where $I$ is an open interval containing $\psi_0$ with $\bar{I}\subset(0,\frac{\pi}{2})$.
Treating the other three families of entering asymptotic vectors similarly, we obtain $\Omega_{\rm in}$ as the union of four sets isomorphic to $\Omega_1$.
Similarly, the set $\Omega_{\rm out}$, isomorphic to $\Omega_{\rm in}$, is obtained by considering the four
families of exiting asymptotic vectors.
It is easy to see that $\Omega$ is transverse to the flow direction.

Denote geodesics in $\cR$ by $\x=\x(t)=(s(t),\theta(t),\psi(t))$ with Clairaut constant $c=c(\x)$.
We parametrise bouncing geodesics~$\x$ taking
$s'(0)=0$ and $\psi(0)=0$ or $\pi$, i.e.\ $\x$ bounces back exactly at time $t=0$. Similarly,
we parametrise crossing geodesics $\x$ taking
$s(0)=0$.

Let $\x=(s(t),\theta(t),\psi(t))$ be a bouncing/crossing geodesic parametrised in this way.
We define the \emph{transition time}
\(
\Upsilon_0(\x)=\min\{t>0: |s(t)|=\ve_1\}.
\)
By symmetry, the transition time of $\x$
from $\Omega_{\rm in}$ to $\Omega_{\rm out}$ is actually equal to $2\Upsilon_0(\x)$ (but it is convenient to call $\Upsilon_0$ the transition time).

Next, we introduce the \emph{transition map} 
\[
f_0:\Omega_{\rm in}\to \Omega_{\rm out}, \quad f_0(x)=g_{2\Upsilon_0(x)}(x)
\]
(where defined).
By \cite[Proposition 4.2]{LMM24},
\[
f_0(-\ve_1,\theta,\psi)=\begin{cases} (-\ve_1,\theta\pm \zeta(\psi),-\psi)
& \text{for bouncing vectors} \\
(\ve_1,\theta\pm \zeta(\psi),\psi)
& \text{for crossing vectors} \end{cases}
\]
where
$\zeta(\psi)  =2\int_{|s(0)|}^{\ve_1} \frac{|c(\x)|}{\xi(s)}\Big[\frac{1+\xi'(s)^2}{\xi(s)^2-c(\x)^2}\Big]^{\frac{1}{2}}ds$.

\subsection{Transition times and derivatives of $f_0$}
\label{ss:times}

As in~\cite{LMM24},
the following partition of $\Omega_{\rm in}$ 
near the set of asymptotic vectors is helpful.

\begin{defn}  \label{defn:Cn}
We define \emph{homogeneity bands} 
$\mfC_n=\mfC_n^{>}\cup\mfC_n^{<}$, $n\ge1$,  
in $\Omega_{\rm in}$ by
\begin{align*}
\mfC_n^{>}&=\big\{x\in \Omega_{\rm in}: 1+\tfrac{1}{(n+1)^2}<|c(x)|< 1+\tfrac{1}{n^2}\big\}\; \text{  (bouncing)};\\
\mfC_n^{<}&=\big\{x\in \Omega_{\rm in}: 1-\tfrac{1}{n^2}< |c(x)|<1-\tfrac{1}{(n+1)^2}\big\}\; \text{ (crossing)}.
\end{align*}
\end{defn}

The next result estimates $\Upsilon_0$ and $\zeta$ in the homogeneity bands. 
Write $\Upsilon_0=\Upsilon_1+\Upsilon_2$ where
$\Upsilon_1$ is the transition time in the portion of the neck between $s=-\ve_1$ and $s=-L$
and $\Upsilon_2$ is the time spent in $\cC$ between $s=-L$ and $s=0$. 

\begin{lemma}\label{lem:Cn}
The following are true.
\begin{enumerate}[i.]
\item[{\rm (1)}] Let $\x$ be a geodesic with entry vector in $\mfC_n$.
Then $\Upsilon_1(\x)\approx n^{\frac{r-2}{r}}$.
If $\x$ is bouncing then $\Upsilon_2(\x)=0$;
if $\x$ is crossing then $\Upsilon_2(\x)\approx n$.
\item[{\rm (2)}] If $\x,\overline{\x}$ are both bouncing or both crossing geodesics, then
$$
|\psi(\Upsilon_0(\x))-\overline{\psi}(\Upsilon_0(\overline{\x}))|\approx |c(\x)-c(\overline{\x})|.
$$
In particular, if the entry vectors of $\x,\overline{\x}$ are both in the same connected component of
$\mfC_n^<$ or of $\mfC_n^>$ then
$$
|\psi(\Upsilon_0(\x))-\overline{\psi}(\Upsilon_0(\overline{\x}))|\leq 2(\ve_1^{2r}+2\ve_1^r)^{-1/2} n^{-3}
+O(n^{-4}).
$$
\item[{\rm (3)}] The following estimates on $\zeta$ hold.
\begin{enumerate}
\item[{\rm (a)}] If $(-\ve_1,\theta,\psi)\in \mfC_n^<$ then
$\zeta'(\psi)\approx -n^3$ and $\zeta''(\psi)\approx n^5$.
\item[{\rm (b)}] If $(-\ve_1,\theta,\psi)\in \mfC_n^>$ then
$\zeta'(\psi)\approx n^{3-\frac{2}{r}}$ and $|\zeta''(\psi)| \ll n^{5-\frac{2}{r}}$.
\end{enumerate}
\end{enumerate}
\end{lemma}

\begin{proof}
(1) By \cite[Lemma 4.4]{LMM24}, $\Upsilon_1\approx n^{\frac{r-2}{r}}$.
Clearly $\Upsilon_2=0$ for bouncing geodesics.
For crossing geodesics, $c=\cos\psi(-\Upsilon_2)$, so
$$
\Upsilon_2=\frac{L}{|\sin\psi(-\Upsilon_2)|}=\frac{L}{\sqrt{1-c^2}}\sim \frac{L}{\sqrt{2}} n.
$$

\noindent
(2) The proof is the same as for \cite[Lemma 4.5]{LMM24}.

\medskip
\noindent
(3) 
The estimates for bouncing geodesics are unchanged from~\cite{LMM24}.
We consider crossing geodesics with
$c(\x)>0$ (the case $c(\x)<0$ is treated analogously).
Let $a=1+\ve_1^r$.
We have $c(\x)=a\cos\psi$, hence $\cos\psi\sim a^{-1}\approx 1$ and
$\sin\psi\sim (1-a^{-2})^{\frac{1}{2}}\approx 1$. Also, $s(0)=0$, so
\begin{align*}
\zeta(\psi) & =
2a \cos\psi\int_0^{\ve_1}  A(s)\big[\xi(s)^2-a^2\cos^2\psi\big]^{-\frac{1}{2}}ds
\end{align*}
where $A(s)=\frac{[1+\xi'(s)^2]^{\frac{1}{2}}}{\xi(s)}\approx 1$. 
By direct calculation,
\begin{align*}
\zeta'(\psi)&=-2a\sin\psi\int_0^{\ve_1} A(s)\xi(s)^2\big[\xi(s)^2-a^2\cos^2\psi\big]^{-\frac{3}{2}}ds\\
&\approx  -\int_0^{\ve_1} [\xi(s)-c(\x)]^{-\frac{3}{2}}ds= -L[1-c(\x)]^{-\frac{3}{2}}-\int_L^{\ve_1} [\xi(s)-c(\x)]^{-\frac{3}{2}}ds
\end{align*}
and
\begin{align*}
\zeta''(\psi) & =2a\cos\psi\int_0^{\ve_1} A(s)\xi(s)^2\big[3a^2\sin^2\psi+a^2\cos^2\psi-\xi(s)^2\big]
\big[\xi(s)^2-a^2\cos^2\psi\big]^{-\frac{5}{2}}ds\\
 &\approx  \int_0^{\ve_1}[\xi(s)-c(\x)]^{-\frac{5}{2}}ds=L[1-c(\x)]^{-\frac{5}{2}}+\int_L^{\ve_1}[\xi(s)-c(\x)]^{-\frac{5}{2}}ds.
\end{align*}
We note that $1-c(\x)\approx n^{-2}$ and $\xi(s)-c(\x)\approx |s-L|^r+n^{-2}$ for $s\in [L,\ve_1]$.
Applying~\cite[Proposition 4.3(1)]{LMM24} with
$\alpha=\frac32$ and $b=n^{-2}$ gives that
$$
\int_L^{\ve_1}[\xi(s)-c(\x)]^{-\frac{3}{2}}ds\approx (n^{-2})^{-\frac32+\frac{1}{r}}=
n^{3-\frac{2}{r}}
$$
and so $\zeta'(\psi)\approx -(n^{-2})^{-\frac32}-n^{3-\frac{2}{r}}=-n^{3}$.

Similarly, applying~\cite[Proposition 4.3(1)]{LMM24} with
$\alpha=\frac52$ and $b=n^{-2}$ gives that
$$
\int_L^{\ve_1}[\xi(s)-c(\x)]^{-\frac{5}{2}}ds\approx (n^{-2})^{-\frac52+\frac{1}{r}}=n^{5-\frac{2}{r}}
$$
and so $\zeta''(\psi)\approx (n^{-2})^{-\frac52}+n^{5-\frac{2}{r}}\approx n^5$.
\end{proof}

\section{Regularity of the stable and unstable foliations}
\label{sec:GW}

For general geodesic flows in nonpositive curvature, the stable/unstable manifolds $\hW^{s/u}_x$ are
$C^1$ and the bundles $\hE^{s/u}_x$ are continuous.
For $C^5$ surfaces, the stable/unstable manifolds $\hW^{s/u}_x$ are
uniformly $C^{1+\frac12}$, see~\cite[Proposition~III]{GerberWilkinson99}.
However, for the Chernov axioms (A5) and (A7) from Appendix~\ref{app:C}, we require extra regularity.

In some situations, it is possible to increase the regularity to
$C^{1+\Lip}$ and H\"older respectively, see~\cite{GerberWilkinson99,LMM24}. But certain key estimates in these references do not hold for surfaces with flat cylinder. (See the discussion in~\cite[Example at end of Section~3]{GerberWilkinson99}.)

In this section, we show that for surfaces with flat cylinder with $r\ge5$, the
stable/unstable manifolds $\hW^{s/u}_x$ are indeed uniformly $C^{1+\Lip}$ as required for Chernov axiom (A5). We also prove for all $r>4$ that the bundles $\hE^{s/u}_x$ have sufficient regularity (even if not H\"older continuous) to  enable verification of Chernov axiom (A7).

\subsection{Curvature estimates}

For general $x\in M=T^1S$, let $p\in S$ denote its footprint and $k_\pm(x)\ge0$ 
the corresponding geodesic curvatures, equal to the curvatures of the curves $\hW^{s/u}_x$ at $x$.
On $T^1\cR$, we use Clairaut coordinates
$x=(p,\psi)=(s,\theta,\psi)\in [-\ve_0,\ve_0]\times\bbS^1\times\bbS^1$.
For $p\in\cN$, write $p=(\pm(L+a),\theta)$ where $a\in[0,\ve_0-L]$.
Throughout this subsection, we only assume $r>4$.
Also, $C>1$ is a constant depending only on the surface $S$. 

To deal with surfaces with flat cylinder, we formulate
the following key lemma with is a suitably weakened analogue of~\cite[Lemma~3.3]{GerberWilkinson99}.

\begin{lemma}  \label{lem:key}
The following hold:
\begin{enumerate}[{\rm (1)}]
\item $k_\pm(x)\le C |\psi|^{(r-2)/r}$ for all $x\in T^1\cC$;
\item $k_\pm(x)\le C \max\{a^{(r-2)/2},\,|\psi|^{(r-2)/r}\}$ for all $x\in T^1\cN$;
\item $k_\pm(x)\ge C^{-1} |\psi|$ for all $x\in T^1\cR$;
\item $k_\pm(x)\ge C^{-1} a^{(r-1)/2}$ for all $x\in T^1\cN$.  
\end{enumerate}
\end{lemma}

\begin{proof}
Our proof largely follows~\cite[Lemma~4.3]{GerberWilkinson99}.\footnote{We note that $m$ in~\cite[Lemma~4.3]{GerberWilkinson99} and our $r$ are related by $r=m+2$. The boundary case is $|\psi(0)|=a^{(r-1)/2}$ compared with
$|\psi(0)|=a^{r/2}$ in~\cite[Lemma~4.3]{GerberWilkinson99}.}

\vspace{1ex} \noindent
The estimates in (1) and~(2) were proved in~\cite[Theorem~3.1]{GerberNitica99}.

\vspace{1ex} \noindent
Next, we consider (3) and~(4) simultaneously with
the convention that $a=0$ for $x\in T^1\cC$.
It suffices to consider $k_+$ since the situation for $k_-$ is identical.

Let $\tS$ be the universal cover of $S$. Lift $x$ to a point in $\tS$, which we will also denote by $x$, and
let $\gamma$ be the geodesic on $\tS$ with $\gamma'(0)=x$. 
Let $u$ be the unstable Riccati solution along $\gamma$, so that $u\ge0$ with $u'=-u^2-K\circ\gamma$
and $k_+(x)=u(0)$.

\vspace{1ex}
Suppose first that $a^{r-1}\le \psi(0)^2$.
Then $a^r\le a\psi(0)^2\le \frac14\psi(0)^2$ for $a$ small.
Without loss of generality, we can assume that $c$ is close to 1 and $\psi(0)$ is close to
zero.
By~\eqref{eq:Clairaut}, the Clairaut constant $c$ satisfies
\begin{align*}
c & =\xi(L+a)\cos\psi(0)=(1+a^r)\cos\psi(0)
\\ &
\le (1+\tfrac14\psi(0)^2)(1-\tfrac12\psi(0)^2+O(\psi(0)^4))
=1-\tfrac14\psi(0)^2+O(\psi(0)^4)<1.
\end{align*}
This means that we are in the crossing case. 
Supposing that the geodesic $\gamma$ has entry vector in $\mfC_n^<$, 
the definition of homogeneity bands in Definition~\ref{defn:Cn}
gives that $c\sim 1-n^{-2}$, so $|\psi(0)|\ll n^{-1}$.
We choose $T>0$ such that $\gamma(-T)$ is leaving $\cR$, with $u(-T)\approx 1$. 
By Lemma~\ref{lem:Cn}(1), $T\le 2\Upsilon_0\ll n$.
Hence by~\cite[Lemma~3.1(v)]{GerberWilkinson99},
\[
k_+(x)=u(0)\ge \frac{u(-T)}{Tu(-T)+1}\approx \frac{1} {T+1} \gg n^{-1} \gg |\psi(0)|\ge a^{(r-1)/2},
\]
yielding the desired lower bounds.

\vspace{1ex} 
It remains to consider the case $\psi(0)^2\le a^{r-1}$.
Let $\gamma(t),\,t\in[-T,0]$, be a maximal geodesic segment lying in
$\left[L+\frac12 a,L+ 2a\right]\times \bbS^1$ with $\gamma(0)\in\{L+a\}\times \bbS^1$ and
$\gamma(-T)\in\{L+\frac12 a,L+2a\}\times \bbS^1$.
We claim that $T\gg   a^{-(r-3)/2}$.
To see this, note that $-K(\gamma(t)))\approx a^{r-2}$ for $t\in[-T,0]$ 
by~\eqref{eq:K}, so 
by~\cite[Lemma~4.2(ii)]{GerberWilkinson99},  $|\psi'(t)|\ll a^{r-1}$ for $t\in[-T,0]$.
Hence $|\psi(t)|\le |\psi(0)|+Ta^{r-1}\ll a^{(r-1)/2}+Ta^{r-1}$.
If $T\le a^{-(r-3)/2}$ then $|\psi(t)|\ll a^{(r-1)/2}$
and so we have the trigonometric estimate
\[
T\ge \tfrac12 a / \max \sin|\psi(t)|\gg a^{-(r-3)/2}.
\]
Hence $T\gg   a^{-(r-3)/2}$.

On the interval $[-T,0]$, 
$u(t)$ satisfies the Riccati equation $u'=-u^2+\alpha$ where $\alpha=-K\circ\gamma$.
Let 
$\alpha_{\ast}=\min_{[-T,0]}\alpha \approx a^{r-2}$. If $u(0)\leq \frac{1}{2} \alpha_{\ast}^{1/2}$, then
$u'(0)\ge \frac34 \alpha_\ast$
and working backwards in time it follows that on $[-T,0]$ we have $u$ non-decreasing, $u(t)\le\frac12 \alpha_\ast^{1/2}\approx a^{(r-2)/2}$ and $u'(t)\ge \frac34 \alpha_\ast\approx a^{r-2}$. Hence 
\[
u(0)\gg u(-T)+Ta^{r-2}
\ge Ta^{r-2}\gg a^{(r-1)/2}.
\]
Otherwise, 
$u(0)\ge \frac12 \alpha_\ast^{1/2}\approx a^{(r-2)/2}$.
Either way, 
$k_+(x)=u(0)\gg a^{(r-1)/2}\ge |\psi(0)|$
as required.
\end{proof}

\begin{cor}  \label{cor:key}
$k_-(x)^{r/(r-2)}\le Ck_+(x)$ and
$k_+(x)^{r/(r-2)}\le Ck_-(x)$
for all $x\in T^1S$.
\end{cor}

\begin{proof}
It suffices to prove the first inequality.
For $p\in\cR$, this follows from Lemma~\ref{lem:key}.
The same argument as in~\cite[Lemma~3.4]{GerberWilkinson99} deals with $p\in S\setminus\cR$.
Indeed, $k_+(x)=0$ only for $x$ tangent to a geodesic along
which $K$ vanishes, so $k_+(x)>0$ for $p\in S\setminus\cC$.
By the continuity of $k_\pm$, 
\[
k_-(x)^{r/(r-2)}\ll k_+(x) \quad\text{for $p\not\in \cR$},
\]
completing the proof.
\end{proof}

\begin{cor}  \label{cor:key2}
The following hold:
\begin{enumerate}[{\rm (1)}]
 \item 
 $|K(p)|\le C  k_+(x)^2$ for $p\in S\setminus \cN$.
 \item
 $|K(p))|\le C  k_+(x)^{2(r-2)/(r-1)}$ for $p\in \cN$.
 \end{enumerate}
\end{cor}

\begin{proof}
Part~(1) is trivial for $p\in\cC$ since $K|_\cC=0$.
The same argument as in~\cite[Lemma~3.4]{GerberWilkinson99} deals with $p\in S\setminus\cR$:
again, $k_+(x)>0$ on $S\setminus\cC$, so
part~(1) follows by continuity of $k_+$ and $K$.

In addition, $K(p)\approx -a^{r-2}$ on $\cN$, so part~(2) follows
from Lemma~\ref{lem:key}(4).
\end{proof}

\begin{lemma} \label{lem:KLip}
The following hold:
\begin{enumerate}[{\rm (1)}]
\item $|K(p_0)-K(p_1)|\le C\big(|K(p_0)|^{1/2}d(p_0,p_1)+d(p_0,p_1)^2\big)$ for all $p_0,p_1\in S$. 
\item $|K(p_0)-K(p_1)|\le C\max\{|K(p_0)|,|K(p_1)|\}^{(r-3)/(r-2)}d(p_0,p_1)$
for all $p_0,p_1\in \cN$.
\end{enumerate}
\end{lemma}

\begin{proof}
Part (1) is~\cite[Lemma~3.2]{GerberWilkinson99}.
To prove part~(2), suppose without loss of generality that $p_0,p_1$ lie in the right-hand portion of $\cN$ 
and write $p_i=(L+a_i,\theta_i)$ where
$0\le a_i\le \ve_0-L$, $\theta_i\in\bbS^1$.
By~\eqref{eq:K}, $-K=\xi''\xi^{-1}(1+(\xi')^2)^{-2}$
so $-K(p_i)\approx a_i^{r-2}$ and
$-K'(p_i)\approx a_i^{r-3}$.
By the mean value theorem,
$|K(p_0)-K(p_1)|\ll \max\{|K'(p_0)|,|K'(p_1)|\}|a_0-a_1|$ and the result follows.~
\end{proof}

\begin{cor} \label{cor:KLip}
Set $q=\min\{1,2(r-3)/(r-1)\}$.
Let $x_0,x_1\in T^1S$ with footprints $p_0,p_1\in S$ and corresponding geodesic 
curvatures $k_+(x_0),k_+(x_1)$. Then
\[
|K(p_0)-K(p_1)|\le C\big\{(k_+(x_0)^q+k_+(x_1)^q)d(p_0,p_1)+d(p_0,p_1)^2\big\}.
\]
\end{cor}

\begin{proof}
If $p_0\not\in\cN$, then it follows from 
Corollary~\ref{cor:key2}(1) and
Lemma~\ref{lem:KLip}(1) that
\[
|K(p_0)-K(p_1)|\ll |K(p_0)|^{1/2}d(p_0,p_1)+d(p_0,p_1)^2 
\ll k_+(x_0)d(p_0,p_1)+d(p_0,p_1)^2,
\]
which implies the required estimate.
The same argument applies if $p_1\not\in\cN$.

If $p_0,p_1\in\cN$, 
then it follows from Corollary~\ref{cor:key2}(2) and Lemma~\ref{lem:KLip}(2) that
\begin{align*}
  |K(p_0) -K(p_1)| & \ll \max\{|K(p_0)|,|K(p_1)|\}^{(r-3)/(r-2)}d(p_0,p_1)
\\ & \ll \max\{k_+(x_0),k_+(x_1)\}^{2(r-3)/(r-1)}d(p_0,p_1)
\end{align*}
which again implies the required estimate.
\end{proof}

\subsection{Two regularity theorems}

\begin{thm} \label{thm:GW1} Assume $r\ge5$.
The curves $\hW_x^{s/u}$ are uniformly $C^{1+\Lip}$.
\end{thm}

\begin{proof}
As in~\cite{GerberWilkinson99}, we focus on the unstable leaves $\hW_x^u$.  The stable case is identical.
Let $\gamma_0$, $\gamma_1$ be two geodesics on $\tS$ such that
$\dist_{\tS}(\gamma_0(t),\gamma_1(t))$
is bounded for $t\leq 0$ and
set $\eps=\dist_\tS(\gamma_0(0),\gamma_1(0))$.
We show that the geodesic curvatures 
$k_+(\gamma_i'(0))$ satisfy
$|k_+(\gamma_0'(0))- k_+(\gamma_1'(0))|\ll \eps$.

By convexity of $\dist_\tS$, we have 
\begin{equation} \label{eq:convex}
\dist_\tS(\gamma_0(t),\gamma_1(t))\le
\dist_\tS(\gamma_0(0),\gamma_1(0))= \eps\qquad\text{for $t\le0$}.
\end{equation}

Let $u_i$ be the unstable Riccati solutions along $\gamma_i$,
so that $u_i\ge0$ and $u_i'=-u_i^2-K_i$ where $K_i=K\circ\gamma_i$.
Then $k_+(\gamma_i'(t))=u_i(t)$ and we obtain 
by~\cite[Lemma~3.1(ii)]{GerberWilkinson99} that
\begin{align} \label{eq:diff}
|k_+(\gamma_0'(0))- &  k_+(\gamma_1'(0))|
 =|u_0(0)-u_1(0)|
\\ & \le \int_{-A}^0|K_0(t)-K_1(t)|\,\hj_0(t)\hj_1(t)\,dt
+|u_0(-A)-u_1(-A)|\,\hj_0(-A)\hj_1(-A), \nonumber
\end{align}
where $\hj_i(t)=\exp\big\{-\int_t^0 u_i(t')\,dt'\big\}\in[0,1]$.
Since $r\ge5$, we have $q=1$ in Corollary~\ref{cor:KLip}.
Taking $A=\eps^{-1}$, 
by Corollary~\ref{cor:KLip} and~\eqref{eq:convex}
\begin{align*}
\int_{-A}^0|K_0-K_1|\, \hj_0\hj_1 & 
 \ll A\eps^2+\eps \int_{-A}^0 (k_+(\gamma_0')+k_+(\gamma_1'))\hj_0\hj_1 
  =\eps+ \eps \int_{-A}^0 (u_0+u_1)\hj_0\hj_1 
 \\ & \le \eps+\eps \int_{-A}^0 (u_0\hj_0+u_1\hj_1) 
 =  \eps+\eps \int_{-A}^0 (\hj_0'+\hj_1') 
\\ & =  \eps+\eps (\hj_0(0)+\hj_1(0)-\hj_0(-A)-\hj_0(-A))
\le   3\eps ,
\end{align*}
which deals with the first term on the right-hand side of~\eqref{eq:diff}.
To estimate the remaining term, we can suppose without loss of generality that $u_0(-A)\ge u_1(-A)$.
By~\cite[Lemma~3.1(iii)]{GerberWilkinson99}, $\hj_0'(-A)\le 1/A$. Hence
\[
|u_0(-A)-u_1(-A)|\,\hj_0(-A)\hj_1(-A)
\le u_0(-A)\hj_0(-A)=\hj_0'(-A)\le 1/A=\eps,
\]
completing the proof.
\end{proof}

\begin{rmk} We also obtain a regularity result for $r\in(4,5)$.
In the first term in~\eqref{eq:diff}, we have to estimate expressions like
 $\int_{-A}^0 u_0^{1/p}\hj_0
 \le \int_{-A}^0 (u_0\hj_0)^{1/p}
 = \int_{-A}^0 (\hj_0')^{1/p}$ 
where $p=(r-1)/(2(r-3))>1$.
Setting $q=(r-1)/(5-r)$ and applying H\"older's inequality,
\[
 \int_{-A}^0 (\hj_0')^{1/p}
\le  \left(\int_{-A}^0 \hj_0'\,dt\right)^{1/p}\,A^{1/q} 
\le   A^{1/q}=A^{(5-r)/(r-1)}.
\]
Taking $A=\eps^{-(r-1)/4}$, we obtain
\[
|k_+(\gamma_0'(0))-k_+(\gamma_1'(0))|
\ll A\eps^2+A^{(5-r)/(r-1)}\eps +A^{-1} \le 3 \eps^{(r-1)/4}.
\]
Hence,
the curves $\hW_x^{s/u}$ are uniformly $C^{1+\frac14 (r-1)}$ which improves the estimate $C^{1+\frac12}$ in~\cite[Proposition~III]{GerberWilkinson99} for all $r\in(4,5)$. However, we do not obtain $C^{1+\Lip}$.
\end{rmk}

We will say that a function $g:M_1\to M_2$ between two metric spaces 
$(M_1,d_1)$ and $(M_2,d_2)$ has {\em $\omega$-logarithmic modulus of continuity} ($\omega>0$) if
there is a constant $C>0$ such that $d_2(g(x),g(y))\le C|\log d_1(x,y)|^{-\omega}$ for all
$x,y\in M_1$ with $d_1(x,y)\le C$.

\begin{thm} \label{thm:GW2}
Assume $r>4$. Then $x\mapsto \hE_x^{s/u}$ has $\omega$-logarithmic modulus of continuity
for all $\omega\in(0,r/2)$.
\end{thm}

\begin{proof}
As in~\cite{GerberWilkinson99}, we reduce to the case where
$x_0,\,x_1\in T^1S$ have the same footprint.  Let
$\theta\ge0$ be the angle between them; without loss of generality $\theta<1$.
Define $T\in(0,\theta^{-1/2}]$ as in~\cite[Eq.~(3.14)]{GerberWilkinson99}. 
Without loss of generality $T>1$.

Let $u_0(t)$, $u_1(t)$ be the corresponding unstable Riccati solutions.
As in the proof of Theorem~\ref{thm:GW1},
\begin{align*} 
|k_+(x_0)- & k_+(x_1)|  =|u_0(0)-u_1(0)|
\\ & \le \int_{-T}^0 |K_0(t)-K_1(t)|\hj_0(t)\hj_1(t)\,dt+ |u_0(-T)-u_1(-T)|\hj_0(-T)\hj_1(-T).
\end{align*}
By Corollary~\ref{cor:KLip},
\begin{align*}
\int_{-T}^0 |K_0-K_1|\hj_0\hj_1
& \ll T\theta+\theta^{1/2}\int_{-T}^0 (k_+(x_0)^q+k_+(x_1)^q)\hj_0\hj_1
\\ & \ll \theta^{1/2}\left\{1+
\int_{-T}^0 (\hj_0')^q+ \int_{-T}^0 (\hj_1')^q\right\}.
\end{align*}
By H\"older's inequality, 
$\int_{-T}^0 (\hj_i')^q\le \left(\int_{-T}^0j_i'\right)^q T^{1-q}
\ll T^{1-q}\ll \theta^{-(1-q)/2}$.
Hence, 
$\int_{-T}^0 |K_0-K_1|\hj_0\hj_1 \ll \theta^{q/2}$ and
\begin{align} \label{eq:k+}   
|k_+(x_0)- k_+(x_1)| 
 \le \theta^{q/2}+|u_0(-T)-u_1(-T)|\hj_0(-T)\hj_1(-T).
\end{align}

As in the proof of Theorem~\ref{thm:GW1}, we thereby obtain
\(
|k_+(x_0)-k_+(x_1)|\ll \theta^{q/2}+T^{-1}.
\)
If $T\ge |\log\theta|^\omega$, then we are finished.
Hence we can suppose from now on that 
\begin{equation} \label{eq:T}
1<T<|\log\theta|^\omega.
\end{equation}
If $k_+(x_i)\le\theta^{q/4}$ for $i=0,1$, then again there is nothing to prove, so we can suppose without loss of generality that
\begin{equation}\label{eq:k++}
 k_+(x_0)\ge \theta^{q/4}.
 \end{equation}

Continuing from~\eqref{eq:k+}, we obtain
\[
|k_+(x_0)-k_+(x_1)|\ll \theta^{q/2}+
\exp\left\{-\int_{-T}^0 u_0\right\}.
\]
By Corollary~\ref{cor:key}, $u_0(-t)\gg u_0(t)^p$ where $p=r/(r-2)$.
Hence, our analogue of~\cite[Eq.~(3.17)]{GerberWilkinson99} is that
there exists $\beta>0$ such that
\[
|k_+(x_0)-k_+(x_1)|  \ll 
\theta^{q/2}+
\exp\left\{-\beta\int_{0}^T u_0^p\right\}
 \le \theta^{q/2}+
\exp\left\{-\beta\int_{1}^T u_0^p\right\}.
\]
The remainder of the argument in~\cite{GerberWilkinson99} shows that
\[
\exp\left\{-\int_{1}^T u_0\right\}\ll k_+(x_0)^{-1}\theta^{q/2}\ll\theta^{q/4},
\]
where in the second inequality we used \eqref{eq:k++}.
(The arguments in~\cite{GerberWilkinson99} are unchanged except for the derivation of~\cite[Eq.~(3.22)]{GerberWilkinson99}.
Instead of applying~\cite[Lemma~3.5]{GerberWilkinson99}, the same argument used to obtain~\eqref{eq:k+} above shows that, in their notation, $|\tilde y(t)|\ll\theta^{q/2}+|\tilde y(1)|$.)
Hence 
\begin{equation} \label{eq:GW2}
\int_1^T u_0\gg |\log\theta|.
\end{equation}

Let $p'=r/2$.
By H\"older's inequality, 
$\int_1^T u_0\le \|u_0\|_p \, T^{1/p'}$, thus by~\eqref{eq:T} and~\eqref{eq:GW2},
\[
\int_1^T u_0^p \ge \left(\int_1^T u_0\right)^pT^{-p/p'}
\gg |\log\theta|^{p(1-\omega/p')}.
\]
This implies that there is a constant $\beta'>0$ such that
\[
|k_+(x_0)-k_+(x_1)|  
\ll  \theta^{q/2}+ \exp\big\{-\beta'|\log\theta|^{p(1-\omega/p')}\big\}
=  \theta^{q/2}+ \exp\big\{-\beta'|\log\theta|^s\big\}
\]
where $s=p(1-\omega/p')>0$ since $\omega<p'=r/2$.
Using the inequality $e^{-x}\ll x^{-m}$ which holds for any fixed $m>0$, we see (taking $m>\omega/s$) that
\[
\exp\big\{-\beta'|\log\theta|^s\big\} \ll 
|\log\theta|^{-sm}\le
|\log\theta|^{-\omega}
\]
completing the proof.
\end{proof}

\section{The first return map $f=g_\Upsilon:\Sigma_0\to\Sigma_0$}
\label{sec:f}

In this section, we construct a first hit Poincar\'e map $f=g_\Upsilon:\Sigma_0\to\Sigma_0$.
The map $f$ has unbounded excursions through the surface of revolution $\cR$ but is bounded elsewhere. 
We then obtain some hyperbolicity properties of $f$ and prove that it satisfies the Chernov axioms listed in Appendix~\ref{app:C},
thereby showing that $f$ is modelled by a Young tower with exponential tails~\cite{Young98}. 

\subsection{The cross-section $\Sigma_0$ and flux measure $\mu_{\Sigma_0}$}
\label{sec:flux}

Roughly speaking, the section $\Sigma_0$ consists of
$\Omega_{\rm in}$ and 
$\Omega_{\rm out}$ together with various local sections situated in the negatively curved region $T^1(S\setminus \cR)$.
However, for technical reasons explained in~\cite{LMM24}, the cross-sections 
$\Omega_{\rm in/out}$ are not suitable to be part of $\Sigma_0$;
instead nearby cross-sections (called $\Sigma_{\rm in,out}$ below)  contained in $g_{[-\chi,\chi]}\Omega_{\rm in/out}\subset T^1\cR$ 
(for $\chi>0$ fixed sufficiently small) are used in the definition of  $\Sigma_0$.
The cross-section $\Sigma_0$ is constructed in such a way that $\inf_{\Sigma_0}\Upsilon>0$ and
$\Upsilon$ is bounded on $\Sigma_0$ except in a neighbourhood of asymptotic vectors.

The precise construction of the section $\Sigma_0$ works as in \cite{LMM24}, since here we also have that
the only source of unboundedness of excursions in $\cR$ occurs near the subset $\{|c|=1\}\cap\Omega_{\rm in}$.
The main properties are that (i)
$\Sigma_0$ is almost perpendicular to the flow direction, and (ii)
a large number of iterates of the singular set under the first return map $f$ to $\Sigma_0$
do not have triple intersections.

Recall that $\mu$ is the smooth probability measure on $M$ induced by the Riemannian metric on $S$.
Locally, it is the product of the area form on $S$ and Lebesgue measure on $\bbS^1$.
The \emph{flux measure}\footnote{This measure is usually not equal to the measure induced by the restriction of the Riemannian metric to $\Sigma_0$.}
is defined on $\Sigma_0$ by the equality
$A(X)=\lim_{t\to 0}t^{-1}\mu(g_{[0,t]}X)$ for $X\subset \Sigma_0$, 
and thus we obtain a smooth probability measure $\mu_{\Sigma_0}=(A(\Sigma_0))^{-1}A$ on $\Sigma_0$.
By definition, given two sections transverse to the flow, the holonomy map in the flow direction
from the first section to the second sends the flux measure of the first to the flux measure of the second.
Therefore, $\mu_{\Sigma_0}$
is invariant under the first return map to $\Sigma_0$. It is for this measure that we will check
the Chernov axioms\footnote{Although not explicitly stated, it was for $\mu_{\Sigma_0}$ that we verified
the Chernov axioms in \cite{LMM24}.}.

Given a compact interval $J\subset (0,\pi)$,
the section $\{-L\}\times \bbS^1\times J$ in Clairaut coordinates is transverse to the flow. Note that $g_t(-L,\theta,\psi)=(-L+t\sin\psi,\theta+t\cos\psi,\psi)$ for $t\ge0$ sufficiently small. Hence
$$
g_{[0,t]}(B\times C)=\bigcup_{\psi\in C}B_{\psi,t}\times\{\psi\}
\quad\text{ for $B\subset\bbS^1$ and $C\subset J$},
$$
where $B_{\psi,t}$ is a parallelogram of base length $|B|$ and height $t\sin\psi$. It follows by Fubini
that the flux measure on $\{-L\}\times \bbS^1\times J$ is equal to $\sin\psi\, d\theta\, d\psi$.

\subsection{The first return map $f=g_\Upsilon$}\label{ss:induced-def}

The \emph{first return time function} of $\Sigma_0$ is 
\[
\Upsilon:\Sigma_0\to (0,\infty], \qquad
\Upsilon(x)=\inf\{t>0:g_tx\in\Sigma_0\}.
\]
By construction, all flow trajectories intersect $\Sigma_0$ infinitely often except those forward and backward asymptotic to the cylinder $\cC$.
We have $\Upsilon(x)=\infty$ if and only if $x$ is an asymptotic vector, hence $\{\Upsilon=\infty\}$ is a finite
union of compact curves, each contained in $g_{[-\chi,\chi]}(\{|c|=1\}\cap\Omega_{\rm in})$.

The \emph{first return map $f=g_\Upsilon:\Sigma_0\to\Sigma_0$} inherits the regularity of  the flow $g_t$; hence it is $C^2$ on $\Sigma_0\setminus\{\Upsilon=\infty\}$ and it has uniformly bounded
derivatives away from $\{\Upsilon=\infty\}$. Since $\Sigma_0$ is almost perpendicular to the flow direction,
the hyperbolicity properties of $f$ away from $\{\Upsilon=\infty\}$ are almost the same as those of the flow.

To maintain control near $\{\Upsilon=\infty\}$, we partition into homogeneity bands,
similarly to Definition~\ref{defn:Cn}.
For this, fix a sufficiently large integer $n_0$
(how large $n_0$ is will depend on a finite number of conditions,
which include the validity of Lemma~\ref{lem:hyperbolicity}(4) and the verification of Chernov axiom~(A8) in Theorem~\ref{thm:Chernov}).

\begin{defn} \label{defn:Dn}
We define \emph{homogeneity bands on $\Sigma_0$}
given by $\mfD_n=\mfD_n^{>}\cup\mfD_n^{<}$, $n\ge n_0$, where
\begin{align*}
\mfD_n^{>}&=\big\{x\in\Int(\Sigma_0)\cap g_{[-\chi,\chi]}\Omega_{\rm in}: 1+\tfrac{1}{(n+1)^2}<|c(x)|< 1+\tfrac{1}{n^2}\big\};\\
\mfD_n^{<}&=\big\{x\in\Int(\Sigma_0)\cap g_{[-\chi,\chi]}\Omega_{\rm in}: 1-\tfrac{1}{n^2}< |c(x)|<1-\tfrac{1}{(n+1)^2}\big\}.
\end{align*}
\end{defn}

\begin{defn}
The \emph{primary singular sets}
$\mfS_{\rm P}^\pm$ are defined as
\begin{align*}
\mfS_{\rm P}^+ & =\{x\in\Sigma_0:\Upsilon(x)<\infty\text{ and }g_{\Upsilon(x)}(x)\in\partial\Sigma_0\}\cup\{\Upsilon = \infty\};
\\[.75ex]
\mfS_{\rm P}^- & =\{x\in\Sigma_0:\Upsilon_-(x)>-\infty\text{ and }g_{\Upsilon_-(x)}(x)\in\partial\Sigma_0\}\cup\{\Upsilon_- = -\infty\};
\end{align*}
where
\(
\Upsilon_-(x)=\sup\{t<0:g_t(x)\in\Sigma_0\}.
\)
The \emph{secondary singular sets}
$\mfS_{\rm S}^\pm$ are
$$
\mfS_{\rm S}^+=\bigcup_{n\geq n_0}\partial\mfD_n
\qquad\text{and}\qquad 
\mfS_{\rm S}^-=\{g_{\Upsilon(x)}(x):x\in\mfS_{\rm S}^+\}.
$$
\end{defn}
Let $\mfS^+=\mfS_{\rm P}^+\cup\mfS_{\rm S}^+$, and
$\mfS^-=\mfS_{\rm P}^-\cup\mfS_{\rm S}^-$.
Then
$f(\Int(\Sigma_0)\setminus\mfS^+)= \Int(\Sigma_0)\setminus\mfS^-$.

Define
\begin{align*}
&\Sigma_{\rm in}=\big\{x\in\Sigma_0\cap g_{[-\chi,\chi]}\Omega_{\rm in}:
\big||c(x)|-1\big|<\tfrac{1}{n_ 0^2}\text{ and }fx\in g_{[-\chi,\chi]}\Omega_{\rm out}\big\}; \\
&\Sigma_{\rm out}=\big\{x\in\Sigma_0\cap g_{[-\chi,\chi]}\Omega_{\rm out}:
\big||c(x)|-1\big|<\tfrac{1}{n_ 0^2}\text{ and }f^{-1}x\in g_{[-\chi,\chi]}\Omega_{\rm in}\big\}.
\end{align*}
Then $f(\Sigma_{\rm in})=\Sigma_{\rm out}$.

The maps
$\mathfrak p_{\rm in}:\Sigma_{\rm in}\to\Omega_{\rm in}$ and $\mathfrak t_{\rm in}:\Sigma_{\rm in}\to [-\chi,\chi]$ are defined by the
equality $z=g_{\mathfrak t_{\rm in}(z)}[\mathfrak p_{\rm in}(z)]$ for $z\in \Sigma_{\rm in}$.
The maps $\mathfrak p_{\rm out}:\Sigma_{\rm out}\to\Omega_{\rm out}$ and
$\mathfrak t_{\rm out}:\Sigma_{\rm out}\to [-\chi,\chi]$ are defined analogously.
They have the same regularity of $g_t$, hence they are $C^2$.
The map $\mathfrak p_{\rm in}$ is surjective. It is also injective,
because if $x\in \Sigma_{\rm in}$ then $x,fx$ are uniquely characterised as being the starting/ending point
of the transition in $\cR$. Therefore, $\mathfrak p_{\rm in}$ is a bijection. 
Since $\Sigma_0$ and $\Omega$ are uniformly transverse to the flow direction,
$\|d\mathfrak p_{\rm in}^{\pm 1}\|\approx 1$.
By symmetry,
the same properties hold for $\mathfrak p_{\rm out}$. 

Recalling the definition of $\mfC_n$ in Definition \ref{defn:Cn}, we note that
$\mfC_n=\mathfrak p_{\rm in}(\mfD_n)$ for all $n\geq n_0$
and $f|_{\Sigma_{\rm in}}=\mathfrak p_{\rm out}^{-1}\circ f_0\circ\mathfrak p_{\rm in}$.

The subspaces $E^{s/u}$ are defined in $\Sigma_{\rm in/out},\,\Omega_{\rm in/out}$ as the projections of
$\hE^{s/u}$ onto the respective tangent spaces, and the maps $\mathfrak p_{{\rm in/out}}$ preserve
these subspaces. Since $\|d\mathfrak p_{{\rm in/out}}^{\pm 1}\|\approx 1$, we obtain that
$\|df|_{E^{s/u}_x}\|\approx \|df_0|_{E^{s/u}_{\mathfrak p_{\rm in}(x)}}\|$
for $x\in\Sigma_{\rm in}$. 

\subsection{Excursion times}\label{ss:excursions}

The first return time $\Upsilon:\Sigma_0\to\R^+$ is bounded on $\Sigma_0\setminus \Sigma_{\rm in}$
and inherits the smoothness of the underlying flow.
Here, we are interested in the regularity of the time
$\Upsilon|_{\Sigma_{\rm in}}$ taken to pass from $\Sigma_{\rm in}$ to $\Sigma_{\rm out}$.

Recall that $2\Upsilon_0$ is the transition time from
$\Omega_{\rm in}$ to $\Omega_{\rm out}$.  
Let $\x=\x(t)$ be a bouncing/crossing geodesic undergoing an excursion in the surface of revolution $\cR$.
Since $f|_{\Sigma_{\rm in}}=\mathfrak p_{\rm out}^{-1}\circ f_0\circ\mathfrak p_{\rm in}$, we have the relation
$\Upsilon(\x)=-\mathfrak t_{\rm in}(x)+2\Upsilon_0(\x)+\mathfrak t_{\rm out}(fx)$,
where $x$ is the starting point of $\x$ in $\Sigma_{\rm in}$.

\begin{lemma} \label{lem:Ups}
Let $\x, \overline{\x}$ be both bouncing or both crossing geodesics with entry vectors $x,\overline{x}\in\Sigma_{\rm in}$
such that $\overline{x}\in W^{s/u}_x$.
Then
$$
|\Upsilon(\x)-\Upsilon(\overline{\x})|\ll d(x,\overline{x})+d(fx,f\overline{x}).
$$
\end{lemma}

\begin{proof}
The proof is the same as for \cite[Lemma 5.11]{LMM24}.
\end{proof}

\subsection{Hyperbolicity properties of $f$ on $\Sigma_{\rm in}$}\label{ss:hyperbolicity-f}

We now establish some hyperbolicity properties of $f|_{\Sigma_{\rm in}}$
with respect to the $\delta$-Sasaki metric $\|\cdot\|$
for a small $\delta>0$. (Recall that this metric is equivalent to the Sasaki metric 
and also to the Clairaut metric.)

A \emph{local unstable manifold (LUM)} is a curve $W\subset \Sigma_0$
such that
$f^{-n}$ is well-defined and smooth on $W$ for all $n\geq 0$, and
$d(f^{-n}x,f^{-n}y)\to 0$ exponentially quickly as $n\to\infty$ for all $x,y\in W$.
We write $W^u_x$ to represent an LUM containing $x$.
Similarly, we define the notion of \emph{local stable manifold (LSM)}
and write $W^s_x$ to represent an LSM containing $x$.

\begin{lemma}\label{lem:hyperbolicity}
Let $r\ge 5$, $\omega\in(1,r/2)$.  The following are true:
\begin{description}
\item[{\rm (1)}] Lebesgue almost every $x\in\Sigma_0$ has an LSM/LUM $W^{s/u}_x$ for $f$.
\item[{\rm (2)}] 
For all $x,\overline{x}\in \Sigma_{\rm in}$,
\begin{enumerate}
\item[{\rm (a)}] There is a continuous function $a:\Sigma_{\rm in}\to\R$ such that
$E^u_x$ is spanned by $\begin{bmatrix}a(x) \\ 1\end{bmatrix}$.
Moreover, 
$|a(x)-a(\bar x)|\ll |\log d(x,\overline{x})|^{-\omega}$. 
\item[{\rm (b)}] There is a $C^{1+{\rm Lip}}$ function $\Theta$ such that
$W^u_x$ is locally the graph $\{(\Theta(\psi),\psi)\}$ of $\Theta$.
\end{enumerate}
\item[{\rm (3)} Growth bounds:] {\color{white} a}
\begin{enumerate}
\item[{\rm (a)}] If $x\in\mfD_n^>$, then $\|df|_{E^u_x}\|\approx n^{3-\frac{2}{r}}$.
\item[{\rm (b)}] If $x\in\mfD_n^<$, then $\|df|_{E^u_x}\|\approx n^3$.
\end{enumerate}
\item[{\rm (4)} Distortion bounds:] If $x,\overline{x}\in\mfD_n^>$ or $x,\overline{x}\in\mfD_n^<$,
with $\overline{x}\in W^u_x$, then
$$
\big|\log \|df|_{E^u_x}\|-\log \|df|_{E^u_{\overline x}}\|\big| \ll
d(fx, f \overline{x})^\frac13.
$$
\item[{\rm (5)} Jacobian of holonomies:]
If $x,\overline{x}\in\mfD_n^>$ or $x,\overline{x}\in\mfD_n^<$,
with $\overline{x}\in W^s_x$, then
\[
\big|\log \|df|_{E^u_x}\|-\log \|df|_{E^u_{\overline{x}}}\|\big|\ll |\log d(x,\overline{x})|^{-\omega}.
\]
\end{description}
\end{lemma}

\begin{proof} (1) This is \cite[Lemma 5.6]{LMM24}, which works equally for surfaces with flat cylinder.

\medskip
\noindent
(2) This is \cite[Lemma 5.7]{LMM24}, which is almost unchanged here. The only difference is the regularity of $a$ which follows from Theorem~\ref{thm:GW2} in place of~\cite[Theorem~2.6]{LMM24}.

\medskip
\noindent
(3) The proof is the same of \cite[Lemma 5.8]{LMM24}, noting that the estimates of $\zeta'$ are given by Lemma~\ref{lem:Cn}(3).
Regarding regularity of $a$, only continuity is used for this estimate.

\medskip
\noindent
(4) This is a version of \cite[Lemma 5.9]{LMM24}, which uses Lemmas~\ref{lem:Cn}(2,3)
and part (2) above. While Lemma \ref{lem:Cn}(3) is different here, it still gives
that $\zeta'\to\infty$ as $n\to\infty$ and $|\zeta''|_\infty/|\zeta'|\ll n^2$
on $\mfD_n$,
which are the necessary ingredients to prove the distortion bounds.

\medskip
\noindent
(5) This is similar to \cite[Lemma 5.10]{LMM24} but using the different regularity property of the function $a$ in part (2) instead of H\"older continuity in~\cite[Lemma 5.7(1)]{LMM24}. 
\end{proof}

We now have all the ingredients to prove that $f$ satisfies the Chernov axioms
(A1)--(A8) listed in Appendix~\ref{app:C}.

\begin{thm}\label{thm:Chernov}
The first return map $f$ satisfies the Chernov axioms (A1)--(A8).
\end{thm}

\begin{proof} 
The verifications are essentially identical to those
in \cite[Section~6]{LMM24}.
The main difference is that
the last line of 
the argument for~(A7) becomes
\begin{align*}
|\log JH(x)| & \ll \sum_{i=0}^\infty |\log d(f^{-i}x,f^{-i}H(x))|^{-\omega}
\le \sum_{i=0}^\infty |\log\{\Lambda^{-i}d(x,H(x))\}|^{-\omega}
\\ & 
= \sum_{i=0}^\infty (i\log\Lambda+\log d(x,H(x))^{-1})^{-\omega}
\ll  1+\sum_{i=1}^\infty i^{-\omega}<\infty.
\end{align*}

\vspace{-5ex}
\end{proof}

\section{Nonstandard CLT for exponential Young towers}
\label{sec:CLT}

As shown in~\cite{LMMsub}, the key step in proving the nonstandard weak invariance principle in Theorem~\ref{thm:WIP} is to establish a nonstandard CLT for certain piecewise constant observables (called $JR_0$ below) related to a discrete first return time function.
In this section we recall, and extend slightly, an abstract argument due to~\cite{BalintChernovDolgopyat11} for proving such a nonstandard CLT.

Let $(Z,\mu_Z)$ be a probability space with an at most countable measurable partition $\{Z_j:\,j\ge1\}$ and let
$F:Z\to Z$ be an ergodic measure-preserving map.
Define the separation time $s(z,z')$ to be the least integer $n\ge0$ such that $F^n z$ and $F^nz'$ lie in distinct partition elements.
We assume that $s(z,z')=\infty$ if and only if $z=z'$; then $d_\theta(z,z')=\theta^{s(z,z')}$ is a metric for $\theta\in(0,1)$.

We say that
$F:Z\to Z$ is a \emph{(full-branch) Gibbs-Markov map} if
(i) $F:Z_j\to Z$ is a measure-theoretic bijection for each $j\ge1$;
(ii) there exists $\theta\in(0,1)$ such that
$\log\xi$ is $d_\theta$-Lipschitz, where
$\xi = d\mu_Z/(d\mu_Z\circ F)$.

Let $\tau:Z\to\Z^+$ be piecewise constant (constant on each $Z_j$) with $\mu_Z(\tau>n)=O(e^{-cn})$ for some $c>0$.
We define the 
\emph{one-sided exponential Young tower}
$\Delta=\{(z,\ell):z\in Z,\,0\le \ell<\tau(z)\}$ and tower map
\[
f_\Delta:\Delta\to \Delta, \qquad
f_\Delta(z,\ell)=\begin{cases} (z,\ell+1), & 0\le \ell<\tau(z)-1 \\
(Fz,0), & \ell=\tau(z)-1.
\end{cases}
\]
An ergodic $f_\Delta$-invariant probability measure is given by
$\mu_\Delta=(\mu_Z\times{\rm counting})/\bar\tau$
where $\bar\tau=\int_Z\tau\,d\mu_Z$.

An observable $V:\Delta\to\R$ is called \emph{piecewise constant} if $V$ is constant on each partition element $Z_j\times\{\ell\}$ where $Z_j$ is a partition element for the Gibbs-Markov map $F:Z\to Z$ and $0\le\ell<\tau|_{Z_j}$.

It is assumed that there is a distinguished integrable piecewise constant observable 
$R_0:\Delta\to\Z^+$. We are interested in piecewise constant observables $JR_0$ where
$J:\Delta{\to\R}$ is bounded and piecewise constant.

\begin{thm} \label{thm:JR}
Let $\{\alpha_i:i\in I\}\subset\R$ be a list of the values attained by $J$.
Let $\sigma_i\ge0$ be constants, $i\in I$, such that 
$\sum_{i\in I}\sigma_i^2\in(0,\infty)$.
Assume that there are constants $C>0$, $\eps\in(0,1)$ such that
\begin{equation} \label{eq:JR1}
 \sum_{i\in I}|\mu_\Delta(R_0=n,\,J=\alpha_i)-2\sigma_i^2 n^{-3}|\le Cn^{-(3+\eps)}
\quad\text{for all $n\ge1$},
\end{equation}
and
\begin{equation} \label{eq:JR2}
\mu_\Delta(R_0=k,\,R_0\circ f_\Delta^n=\ell)\le Ck^{-(2+\eps)}\ell^{-(2+\eps)}
\quad\text{for all $k,\ell,n\ge1$}.
\end{equation}
Then $JR_0$ satisfies a nonstandard CLT with variance 
$\sigma_J^2=\sum_{i\in I}\alpha_i^2\sigma_i^2\ge0$. That is,
$(n\log n)^{-1/2}\sum_{j=0}^{n-1}(JR_0)\circ f_\Delta^j\to_d N(0,\sigma_J^2)$.
\end{thm}

In the remainder of this section, we prove Theorem~\ref{thm:JR}
following~\cite{BalintChernovDolgopyat11}.
Let $\sigma_{R_0}^2=\sum_{i\in I}\sigma_i^2\in(0,\infty)$.

\begin{prop} \label{prop:R0}
$\mu_\Delta(R_0=n) =2\sigma_{R_0}^2 n^{-3} + O(n^{-(3+\eps)})$.
\end{prop}

\begin{proof}
This is immediate from condition~\eqref{eq:JR1}.
\end{proof}

\begin{prop} \label{prop:JR}
$\int_{R_0\le p}(JR_0)^2\,d\mu_\Delta=2\sigma_J^2\log p +O(1)$ for $p\ge1$.
\end{prop}

\begin{proof}
By condition~\eqref{eq:JR1},
\begin{align*}
\int_{R_0\le p}(JR_0)^2\,d\mu_\Delta
& =\sum_{m=1}^p\sum_{i\in I} \alpha_i^2m^2\mu_\Delta(R_0=m,\,J=\alpha_i)
\\ & =\sum_{m=1}^p m^2\left(\sum_{i\in I}2\alpha_i^2\sigma_i^2m^{-3}+O(m^{-(3+\eps)})\right)
=2\sigma_J^2\log p+O(1)
\end{align*}
as required.
\end{proof}

For $1\le d\le e\le\infty$, define
\[
A_{d,e}=\left(JR_0-\frac{1}{\mu_\Delta(d\le R_0<e)}\int_{d\le R_0<e}JR_0\,d\mu_\Delta\right)1_{\{d\le R_0<e\}}.
\]

\begin{lemma} \label{lem:A}
There are constants $C>0$ and $\gamma\in(0,1)$,
independent of $n$, $d$, $e$, such that
\[
\left|\int_{\Delta} A_{d,e}\,(A_{d',e'}\circ f_\Delta^n)\,d\mu_\Delta\right|\le C\gamma^n
\]
for all $n\ge1$, $1\le d\le e\le \infty$, 
$1\le d'\le e'\le \infty$.
\end{lemma}

\begin{proof}
We claim that 
\begin{equation} \label{eq:Abound}
|A_{d,e}1_{\{R_0=k\}}|\ll k.
\end{equation}
We begin by verifying this claim.

For the first term in $A_{d,e}1_{\{R_0=k\}}$, we have
$|JR_0 1_{\{R_0=k\}}|\le |J|_\infty\,k$.
The remaining term is bounded by $|J|_\infty B$, where 
\[
B=1_{\{R_0=k,\,d\le R_0<e\}}\,\mu_\Delta(d\le R_0<e)^{-1}\int_{d\le R_0<e}R_0\,d\mu_\Delta. 
\]
Note that if $B\neq0$, then $d\le k<e$, so we can suppose throughout that $d\le k<e$.
Also, $B\in[d,e]$ so for $e\le 2d$ we obtain $B\le e\le 2d\le 2k$.
Hence we can suppose that $e>2d$.

By Proposition~\ref{prop:R0}, 
\[
\int_{d\le R_0<e}R_0\,d\mu_\Delta\le  \sum_{m\ge d}m\mu_\Delta(R_0=m)\ll d^{-1}.
\]
Also, there exists $m_0\ge1$ such that $\mu_\Delta(R_0=m)>\sigma_{R_0}^2 m^{-3}$ for all $m\ge m_0$.
Hence, for $d\ge m_0$, we have
$\mu_\Delta(d\le R_0<e)\gg d^{-2}-e^{-2}\ge \frac{3}{4}d^{-2}$, and we 
obtain $B\ll d\le k$ as required.

On the other hand, if $d\le m_0<e$, then $\mu_\Delta(d\le R_0<e)\ge \mu_\Delta(R_0= m_0)>0$ and we obtain $B\ll d^{-1}\le 1\le  k$.
This leaves the case where $e\le m_0$ for which we have the trivial bound $B\le m_0\le m_0k$.  This completes the proof of estimate~\eqref{eq:Abound}.

Now set $B_1=A_{d,e}$, $B_2=A_{d',e'}$.
By~\eqref{eq:Abound},
\begin{equation} \label{eq:Bbound}
|B_i 1_{\{R_0\le N\}}|\ll N.
\end{equation}

Since $\int B_i\,d\mu_\Delta=0$,
\[
\int_{\Delta}B_1 1_{\{R_0\le N\}}\, (B_2 1_{\{R_0\le N\}})\circ f_\Delta^n\,d\mu_\Delta
=I_1+I_2
\]
where
\begin{align*}
I_1 & =\int_{\Delta}\big(B_1 1_{\{R_0\le N\}}-{\SMALL\int} B_1 1_{\{R_0\le N\}}\,d\mu_\Delta\big)\, \big(B_2 1_{\{R_0\le N\}}-{\SMALL\int} B_2 1_{\{R_0\le N\}}\,d\mu_\Delta\big)\circ f_\Delta^n\,d\mu_\Delta,
\\ I_2 & =\int_{\Delta}B_1 1_{\{R_0> N\}}\,d\mu_\Delta
\int_{\Delta}B_2 1_{\{R_0> N\}}\,d\mu_\Delta.
\end{align*}
The functions $B_i 1_{\{R_0\le N\}}$ are piecewise constant.
Hence it follows from~\eqref{eq:Bbound} together with exponential decay of correlations for dynamically H\"older observables that there exists $c>0$ such that
\[
|I_1|\ll e^{-cn}|B_1 1_{\{R_0\le N\}}|_\infty
|B_2 1_{\{R_0\le N\}}|_\infty \ll e^{-cn}N^2.
\]
Also, by~\eqref{eq:Bbound} and Proposition~\ref{prop:R0},
\[
\left|\int B_i 1_{\{R_0>N\}}\,d\mu_\Delta\right|
\le \sum_{k>N}\int |B_i| 1_{\{R_0=k\}}\,d\mu_\Delta
\ll \sum_{k>N} k\mu_\Delta(R_0=k)\ll N^{-1}.
\]
Hence, $|I_2|\ll N^{-2}$ and
\begin{equation} \label{eq:B1}
\left|\int_{\Delta}B_1 1_{\{R_0\le N\}}\, (B_2 1_{\{R_0\le N\}})\circ f_\Delta^n\,d\mu_\Delta\right|\ll e^{-cn}N^2 + N^{-2}.
\end{equation} 

Next, by condition~\eqref{eq:JR2} and estimate~\eqref{eq:Bbound},
\begin{align*}
\left|\int_{\Delta}B_1 1_{\{R_0\le N\}}\, (B_2 1_{\{R_0>N\}}) \circ f_\Delta^n\,d\mu_\Delta\right| & \le
\sum_{k=1}^N \sum_{\ell=N+1}^\infty 
\left|\int_{\Delta}B_1 1_{\{R_0=k\}}\, (B_2 1_{\{R_0=\ell\}}) \circ f_\Delta^n\,d\mu_\Delta\right| 
\\ & \ll
\sum_{k=1}^N \sum_{\ell=N+1}^\infty k\ell \mu_\Delta(R_0=k,\,R_0\circ f_\Delta^n=\ell)
\\ & \ll 
\sum_{k=1}^N \sum_{\ell=N+1}^\infty k^{-(1+\eps)}\ell^{-(1+\eps)} 
\ll N^{-\eps}.
\end{align*}
Similarly,
\begin{align*}
\left|\int_{\Delta}B_1 1_{\{R_0>N\}}\, B_2 \circ f_\Delta^n\,d\mu_\Delta\right| & \ll
\sum_{k=N+1}^\infty\sum_{\ell=1}^\infty k\ell \mu_\Delta(R_0=k,\,R_0\circ f_\Delta^n=\ell)
\ll N^{-\eps}.
\end{align*}
Combining these with~\eqref{eq:B1} and taking
$N=[e^{cn/(2+\eps)}]$, we obtain that
\[
\left|\int_{\Delta}B_1 \, (B_2 \circ f_\Delta^n)\,d\mu_\Delta\right|
\ll e^{-cn}N^2+N^{-\eps} \ll e^{-\eps cn/(2+\eps)},
\]
as required.
\end{proof}

\begin{pfof}{Theorem~\ref{thm:JR}}
Following~\cite[Section~3]{BalintChernovDolgopyat11}, we take $p=n^{1/2}(\log n)^{-11}$ and
$q=n^{1/2}\log\log n$.
Using Proposition~\ref{prop:R0}, the nonstandard CLT for $JR_0$ 
reduces to the nonstandard CLT for $JR_01_{\{R_0\le q\}}$.
Then, using Lemma~\ref{lem:A}, it is shown in~\cite[Section~3]{BalintChernovDolgopyat11} how to reduce further from $JR_01_{\{R_0\le q\}}$ to $\hat A=JR_01_{\{R_0\le q\}}-A_{p,q}$.

Finally, the nonstandard CLT for $\hat A$ follows by the argument in~\cite[Section~4]{BalintChernovDolgopyat11}.
The ingredients are:
\begin{enumerate}[$\bullet$]
\item An exponential multiple decorrelations bound for dynamically H\"older observables (\cite[Lemma~4.1]{BalintChernovDolgopyat11}) which is automatically satisfied for systems modelled by Young towers with exponential tails~\cite[Theorem~7.4.1]{ChernovMarkarian};
\item A locally H\"older estimate on $JR_0$ (\cite[Lemma~4.2]{BalintChernovDolgopyat11}) which is trivially satisfied here (with $d=1$) since $JR_0$ is piecewise constant;
\item Proposition~\ref{prop:JR} (which corresponds to~\cite[Lemma~4.3]{BalintChernovDolgopyat11}).
\end{enumerate}

\vspace{-4ex}
\end{pfof}

If we already know that $(f_\Delta,\mu_\Delta)$ satisfies the Chernov axioms, then 
checking condition~\eqref{eq:JR2} can be done following closely the argument in the proof of~\cite[Lemma~3.2]{BalintChernovDolgopyat11}:

\begin{lemma} \label{lem:abstract-pairs}
Assume that $(f_\Delta,\mu_\Delta)$ satisfies the Chernov axioms and that there are constants $a,b,c\geq 0$ with $a\ge c$ such that:
\begin{enumerate}[i.]
\item[{\rm (i)}] $\{R_0=k\}$ is foliated by unstable curves of size $\approx k^{-a}$
and by stable curves of size $O(k^{-b})$.
\item[{\rm (ii)}] (Growth-type lemma) $f_\Delta$ expands unstable curves inside $\{R_0=k\}$ by a factor $\approx k^c$.
\end{enumerate}
Then 
\(
\mu_\Delta(R_0=k,\,R_0\circ f_\Delta^n =\ell) \ll k^{-(b+c)}\ell^{-a}
\)
for all $k,\ell,n\ge1$.
\end{lemma}

\begin{proof}
As mentioned, we follow~\cite[Lemma 3.2]{BalintChernovDolgopyat11}, noting that
our set $\{R_0=k\}$ is equivalent to their set $\mathcal M_k$. 
By (ii),
$f_\Delta$ expands $\{R_0=k\}$ in the unstable direction by a factor of the order of $k^c$,
 hence by~(i), $f_\Delta\{R_0=k\}$ is a set foliated by unstable
curves of size $\approx k^{c-a}$. Therefore the restriction of
$\mu_{\Delta}$ to $f_\Delta\{R_0=k\}$ can be represented by a standard family 
$\mfG_k$ with $Z$-function $Z(\mfG_k)=O(k^{a-c})$; 
see \cite[Sect.~7.4]{ChernovMarkarian} 
for the definitions and basic properties
of standard families and $Z$-functions. In particular, the $Z$-functions of the 
images $f_\Delta^n\mfG_k$ are also standard families and satisfy a Lasota-Yorke type inequality
$$
Z(f_\Delta^n\mfG_k)\leq c_1 \vartheta^n Z(\mfG_k)+c_2
$$ 
for some constants $c_1,c_2>0$ and $\vartheta\in(0,1)$. Hence $Z(f_\Delta^n\mfG_k)=O(k^{a-c})$
uniformly in $k,n$. 

Denote the $u$-size of $\{R_0=\ell\}$ by $u_\ell\approx \ell^{-a}$.
By \cite[Sect.~7.4]{ChernovMarkarian},
\[
\mu_\Delta(R_0=k,R_0\circ f_\Delta^n=\ell)
\ll \mu_\Delta(R_0=k)\cdot Z(f_\Delta^n\mfG_k)\cdot  u_\ell 
=O(k^{-(b+c)}\ell^{-a})
\]
as required.
\end{proof}

\section{Nonstandard WIP for surfaces with flat cylinder}
\label{sec:WIP}

In this section, we prove Theorems~\ref{thm:CLT} and~\ref{thm:WIP}. In particular, we prove a nonstandard WIP for the geodesic flow
$g_t:M\to M$ on a surface $S$ with flat cylinder for all $r\ge5$.

In Subsection~\ref{sec:g}, we construct a new first return map $g=g_h:\Sigma\to\Sigma$ where the first return time function $h:\Sigma\to\R^+$ is bounded above and below and H\"older.
To distinguish it from $f$, we call $g$ the \emph{global Poincar\'e map}.
In Subsection~\ref{sec:R}, we obtain asymptotics for the tail of the discrete first return time function $R:\Sigma_0\to\Z^+$ to $\Sigma_0$ under $g$.
In Subsection~\ref{sec:pfWIP}, we complete the proof of Theorems~\ref{thm:CLT} and~\ref{thm:WIP}.

\subsection{Global Poincar\'e map $g=g_h:\Sigma\to\Sigma$ with bounded $h$}
\label{sec:g}

In Section~\ref{sec:f}, we defined a uniformly hyperbolic first return map 
$f=g_\Upsilon:\Sigma_0\to\Sigma_0$.
The return time $\Upsilon:\Sigma_0\to(0,\infty)\cup\{\infty\}$ satisfies 
$\inf \Upsilon>0$ but is unbounded.

As in~\cite{LMM24}, we adjoin a section $\Omega_0\subset T^1\cR$ 
defined in a neighbourhood of the curve $\Xi([-L,L]\times\{0\})$ given by a disjoint union of two open disks
$$
\Omega_0=(-L-\chi,L+\chi)\times\{0\}\times\big\{(-\chi,\chi)\cup (\pi-\chi,\pi+\chi)\big\}.
$$
Define $\Sigma=\Sigma_0\cup\Omega_0$.
In this way, we obtain a global Poincar\'e map 
\[
g=g_h:\Sigma\to\Sigma,  \qquad  g(x)=g_{h(x)}(x),
\]
which is no longer uniformly hyperbolic but for which the 
return time $h:\Sigma\to(0,\infty)$ is bounded above and H\"older with $\inf h>0$.
We note that $h$ coincides with $\Upsilon$ on $\Sigma_0\setminus\Sigma_{\rm in}$.
As in Section~\ref{sec:g}, we define the normalised flux measure $\mu_\Sigma$ on $\Sigma$, 
giving a $g$-invariant probability measure.

\subsection{Discrete first return time function $R:\Sigma_0\to\Z^+$}
\label{sec:R}

Write $f=g^R$ where
\[
R:\Sigma_0\to\Z^+, \qquad R(x)=\inf\{n\ge1: g^nx\in\Sigma_0\}.
\]
Then $R$ is bounded on $\Sigma_0\setminus\Sigma_{\rm in}$ but unbounded on $\Sigma_{\rm in}$; indeed $R=\Upsilon+O(1)$ since $h$ is bounded above and below.
Moreover,
\[
R=R_\cC+R_\cN+O(1) =R_\cC+O(R_\cN)
\]
where $R_\cC$ is the transition time in the cylinder $\cC$ and
$R_\cN$ is the time spent in the neck $\cN$.

Recall from Section~\ref{sec:flux}
the notion of flux measure $A$ 
which defines the $f$-invariant probability measure $\mu_{\Sigma_0}=\frac{1}{A(\Sigma_0)}A$ on $\Sigma_0$.
The dynamical definition implies that if we are given 
two sections, then the holonomy map between sections sends the flux measure on the first
section to the flux measure on the second section.
Set
\[
\sigma_R^2= 
\frac{4L^2}{A(\Sigma_0)\pi}.
\]

\begin{lemma} \label{lem:RC}
\begin{enumerate}[{\rm (1)}]
\item $\mu_{\Sigma_0}(R_\cC=n)= 2\sigma_R^2 n^{-3}+O(n^{-4})$.
\item There exists $p>2$ such that $R-R_\cC\in L^p$.
\end{enumerate}
\end{lemma}

\begin{proof}
(1)
Clearly, $R_\cC=0$ in the bouncing case. 
In the crossing case, there are four possibilities: trajectories may traverse the cylinder from left to right or vice versa and may spiral clockwise or counterclockwise.
Since geodesics are spiralling around a cylinder of length $2L$ and circumference $2\pi$, we see that $R_\cC=n$ if and only if 
$\tan\psi\in \pm I$ where $I=\Big(\frac{L}{(n+1)\pi},\frac{L}{n\pi}\Big]$.
Recalling that the flux measure on $\{-L,L\}\times \bbS^1\times (0,\pi)$ is $\sin\psi\, d\theta\, d\psi$,
\[
\mu_{\Sigma_0}(R_\cC=n)=\frac{A(R_\cC=n)}{A(\Sigma_0)}
=\frac{1}{A(\Sigma_0)}\cdot 4\cdot 2\pi \cdot \int_{\tan\psi\in I}\sin\psi\,d\psi
=\frac{8L^2}{A(\Sigma_0)\pi}\,\frac{1}{n^3} + O(n^{-4}).
\]

\vspace{1ex} \noindent
(2) 
Let $\Upsilon_\cN$ be the time spent by a geodesic in the neck $\cN$ during an excursion in the surface of revolution $\cR$,
	(so $\Upsilon_\cN=2\Upsilon_1$ in Lemma~\ref{lem:Cn}).
By~\cite[Lemma~5.4]{LMM24}, 
$\mu_{\Sigma_0}(\Upsilon_\cN>n)\approx n^{-2r/(r-2)}$.
Hence $R_\cN=\Upsilon_\cN+O(1)\in L^p$ for all $p<2r/(r-2)$.
\end{proof}

\begin{lemma} \label{lem:pairs}
$\mu_{\Sigma_0}(R_\cC=k,\,R_\cC\circ f^n =\ell) \ll  k^{-3}\ell^{-3}$
for all $k,\ell,n\ge1$.
\end{lemma}

\begin{proof}
By Theorem~\ref{thm:Chernov},
$f:\Sigma_0\to\Sigma_0$ satisfies the Chernov axioms. 
Hence we are able to apply Lemma~\ref{lem:abstract-pairs}.
We take $b=0$. It remains to show that we can take $a=c=3$ in Lemma~\ref{lem:abstract-pairs}.

For geodesics starting in $\mfC_n^<$,
we have $\cos\psi=1-n^{-2}+O(n^{-3})$, so $\tan\psi=\sqrt{2}\,n^{-1}+O(n^{-2})$.
Since geodesics are spiralling around a cylinder of length $2L$ and circumference $2\pi$, it holds that
$|R_\cC-\frac{L}{\pi \tan\psi}|\le1$.
Hence $R_\cC|_{\mfC_n^<} =\frac{L}{\sqrt 2\,\pi}n+O(1)$.
It follows that $\{R_\cC=k\}$ is contained in a union of homogeneity bands $\mfC_n^<$
with
$n\approx k$.

By Lemma~\ref{lem:Cn}(3), 
$|\zeta'(\psi)|\approx n^3\approx k^3$ on $\mfC_n^<$.
This verifies that $c=3$ in Lemma~\ref{lem:abstract-pairs}.
By Lemma~\ref{lem:Cn}(3), unstable curves are aligned parallel to the $\psi$-axis. Hence, by definition of $\mfC_n$, 
unstable curves have size $\approx n^{-3}\approx k^{-3}$ in $\mfC_n$.
This verifies that $a=3$ in Lemma~\ref{lem:abstract-pairs}.
\end{proof}

By Theorem~\ref{thm:Chernov},
the uniformly hyperbolic first return map $f:\Sigma_0\to\Sigma_0$ satisfies the Chernov axioms and hence, by~\cite[Theorem~3.1]{LMM24}, 
is modelled by a Young tower with exponential tails~\cite{Young98}.
The discrete return time $R$ lifts to the tower and is piecewise constant, so is well-defined and piecewise constant on the one-sided Young tower 
$(\Delta,f_\Delta,\mu_\Delta)$ as in Section~\ref{sec:CLT}.

\begin{cor} \label{cor:RCLT}
$R:\Delta\to\Z^+$ satisfies a nonstandard CLT 
on $(\Delta,f_\Delta,\mu_\Delta)$. That is, 
$(n\log n)^{-1/2}(\sum_{j=0}^{n-1}R\circ f_\Delta^j-n\bar R)\to_d N(0,\sigma_R^2)$ where
$\bar R=\int_\Delta R\,d\mu_\Delta$.
\end{cor}

\begin{proof} 
By Lemma~\ref{lem:RC}(1), condition~\eqref{eq:JR1} holds (with $J\equiv1$, $\eps=1$ and $R_0=R_\cC$).
Similarly, condition~\eqref{eq:JR2} holds by Lemma~\ref{lem:pairs}.
It thereby follows from Theorem~\ref{thm:JR} that 
$R_\cC$ satisfies a nonstandard CLT with variance $\sigma_R^2$.

The following standard argument shows that this limit theorem for $R_\cC$ is inherited by $R$.
By Lemma~\ref{lem:RC}(2), $R=R_\cC+H$ where $H\in L^p$ for some $p>2$.
Let $H_\tau=\sum_{\ell=0}^{\tau-1}H\circ f_\Delta^\ell:Z\to\R$.
Since $\tau$ has exponential tails, $R_\tau\in L^q$ for $q\in(2,p)$.
Also, $R_\tau$ is piecewise constant so it is well-known (see for example~\cite[Corollary~2.5]{MN05}) that 
$R_\tau$ satisfies a standard CLT on $(F,Z,\mu_Z)$ (with normalisation $n^{1/2}$).
By~\cite{MT04}, $H$ satisfies a standard CLT on $(f_\Delta,\Delta,\mu_\Delta)$.
Hence 
$(n\log n)^{-1/2}(\sum_{j=0}^{n-1}H\circ f_\Delta^j-n\bar H)\to_d 0$ where
$\bar H=\int_\Delta H\,d\mu_\Delta$.
The result follows.
\end{proof}

\subsection{Proof of (non)standard CLT and WIP for the flow}
\label{sec:pfWIP}

In this subsection, we complete the proof of 
Theorems~\ref{thm:CLT} and~\ref{thm:WIP}.
Let $v:M\to\R$ be a H\"older observable with $\int_M v\,d\mu=0$.
Define induced observables
\[
v_h:\Sigma\to\R, \qquad V:\Sigma_0\to\R,
\]
\[
v_h=\int_0^h v\circ g_t\,dt,
\qquad
V=\sum_{j=0}^{R-1} v_h\circ g^j.
\]
Define
\begin{equation} \label{eq:alpha}
\alpha_0=\frac{1}{2L}\int_{\bbS^1}\int_{-L}^L v(s,\theta,0)\,ds\,d\theta, \qquad
\alpha_\pi=\frac{1}{2L}\int_{\bbS^1}\int_{-L}^L v(s,\theta,\pi)\,ds\,d\theta.
\end{equation}

\begin{lemma}  \label{lem:V}
There exists $p>2$ such that
\[
V-J R_\cC \in L^p
\]
where $J$ takes the values $\alpha_0$ and $\alpha_\pi$
each with probability $\frac12$ on $(\Sigma_0,\mu_{\Sigma_0})$.
\end{lemma}

\begin{proof}
Since $v$ and $h$ are H\"older and $g_t$ is smooth, there exists $\eta>0$ such that $v_h$ is $C^\eta$.
We claim that $V=JR_\cC+O(R_\cC^{1-\eta})+O(R_\cN)$.
By Lemma~\ref{lem:RC}(2), 
$R_\cN\in L^p$ for some $p>2$.
By Lemma~\ref{lem:RC}(1), $R_\cC\in L^q$ for all $q<2$ so $R_\cC^{1-\eta}\in L^p$ for some (possibly smaller) $p>2$, proving the result.

It remains to prove the claim.
Recall that the section $\Omega_0\subset \cR$ is contained in $\{\theta=0\}$ and is a neighbourhood therein of $\{s\in[-L,L],\,\psi=0\}\cup\{s\in[-L,L],\,\psi=\pi\}$. Accordingly, we write
$\Omega_0\cap\cC=\Omega_{0,0}\cup \Omega_{0,\pi}$.

For $\chi>0$ chosen small enough in the definition of $\Omega_0$, we can ensure that geodesics crossing $\cR$ can intersect at most one of  $\Omega_{0,0}$ and $\Omega_{0,\pi}$.
For $p\in\Omega_{0,0}$, let $p_j$, $j=0,\dots,R_\cC$, be the iterates of $p$ under $g$ that lie in $\Omega_{0,0}$.
Then $p_j=(s_j,0,\psi_0)$ where 
\[
\tan\psi_0=L(\pi R_\cC)^{-1}+O(R_\cC^{-2}),
\]
 and 
$R_\cC=[L/(\pi\tan\psi_0)]$.
Moreover, $s_j=s_0+2\pi j \tan\psi_0$ for each $j=1,\dots,R_\cC$.

Since $v_h$ is bounded, 
\[
V(p) =\sum_{j=1}^{R_\cC} v_h(p_j) + O(R_\cN).
\]
Since $v_h$ is $C^\eta$,
\[
|v_h(p_j)-v_h(s_j,0,0)|=  |v_h(s_j,0,\psi_0)-v_h(s_j,0,0)|\ll
|\psi_0|^\eta \ll R_\cC^{-\eta}.
\]
It follows that
\[
V(p) =\sum_{j=1}^{R_\cC} v_h(s_j,0,0) + O(R_\cC^{1-\eta}) + O(R_\cN) 
\]
so it remains to estimate 
$\sum_{j=1}^{R_\cC} v_h(s_j,0,0)$.

Let $\overline{S}$ and $\underline{S}$ be upper and lower Riemann sums for 
$\int_{-L}^L v_h(s,0,0)\,ds$
with respect to the subdivision $-L\le s_0<\dots<s_{R_\cC}\le L$.
Since $s_j-s_{j-1}\ll R_\cC^{-1}$ and
$v_h$ is $C^\eta$,
$\overline{S}-\underline{S}\ll R_\cC^{-\eta}$.
Moreover,
\[
\underline{S}\le \int_{-L}^L v_h(s,0,0)\,ds \le \overline{S},
\qquad  \underline{S} \le \sum_{j=1}^{R_\cC} v_h(s_j,0,0)(s_j-s_{j-1}) \le \overline{S},
\]
so
\[
\left|\int_{-L}^L v_h(s,0,0)\,ds - \sum_{j=1}^{R_\cC} v_h(s_j,0,0)(s_j-s_{j-1})\right|
\le \overline{S}-\underline{S}
\ll R_\cC^{-\eta}.
\]
We deduce that 
\begin{align*}
\sum_{j=1}^{R_\cC} v_h(s_j,0,0)
 & =(2\pi\tan\psi_0)^{-1}\sum_{j=1}^{R_\cC} v_h(s_j,0,0)(s_j-s_{j-1}).
\\
 & =(2\pi\tan\psi_0)^{-1}\left[\int_{-L}^L v_h(s,0,0)\,ds+O(R_\cC^{-\eta})\right]
\\ &
=R_\cC\frac{1}{2L}\int_{-L}^L v_h(s,0,0)\,ds+O(R_\cC^{1-\eta}) .
\end{align*}
Moreover,
\[
v_h(s,0,0)=\int_0^{h(s,0,0)}(v\circ g_t)(s,0,0)\,dt=\int_{\bbS^1}v(s,\theta,0)\,d\theta,
\]
completing the estimate for geodesics intersecting $\Omega_{0,0}$.
By time-reversibility, the calculation for 
geodesics intersecting $\Omega_{0,\pi}$ is identical and the two cases are equally likely.
\end{proof}

As in the introduction, we set $I_v=\alpha_0^2+\alpha_\pi^2$.
Define 
\[
\sigma_J^2=\tfrac12\sigma_R^2I_v,
\qquad
\sigma_v^2=\bar h^{-1}\bar R^{-1}\sigma_J^2.
\]
where $\bar R=\int_{\Sigma_0}R\,d\mu_{\Sigma_0}$ and
$\bar h=\int_\Sigma h\,d\mu_\Sigma$.

\vspace{1ex}
\begin{pfof}{Theorems~\ref{thm:CLT} and~\ref{thm:WIP}}
Since $f$ is modelled by a
Young tower with exponential tails,
we can apply~\cite[Theorem~5.1]{LMMsub}
(with $M$, $\Sigma$ and $\Sigma_0$ playing the roles of $\Gamma^h$, $\Gamma$ and $\Delta$ respectively).

By Lemma~\ref{lem:RC},
\[
\mu_{\Sigma_0}(R_\cC=n,J=\alpha_0)= \mu_{\Sigma_0}(R_\cC=n,J=\alpha_\pi)=\tfrac12 \mu_{\Sigma_0}(R_\cC=n)=
\sigma_R^2 n^{-3}+O(n^{-4}).
\]
Hence, condition~\eqref{eq:JR1} holds (with $\eps=1$ and $R_0=R_\cC$).
Similarly, condition~\eqref{eq:JR2} holds by Lemma~\ref{lem:pairs}.
It thereby follows from Theorem~\ref{thm:JR} that 
$JR_\cC$ satisfies a nonstandard CLT with variance 
$\sigma_J^2\ge0$.
Also, $R$ satisfies a nonstandard CLT with variance $\sigma_R^2$ by Corollary~\ref{cor:RCLT}.
By Lemma~\ref{lem:V}, the hypotheses of~\cite[Theorem~5.1]{LMMsub} are
satisfied.

Notice that $I_v=0$ if and only if $J\equiv0$, equivalently $\sigma_J^2=0$. 
If $I_v\neq0$, then we are in the situation of~\cite[Theorem~5.1(a)]{LMMsub},
completing the proof of Theorem~\ref{thm:WIP}. This implies the nonstandard CLT in the first statement of Theorem~\ref{thm:CLT}.

If $I_v=0$, then we are in the situation of~\cite[Theorem~5.1(b)]{LMMsub}, proving the standard CLT in the second statement of Theorem~\ref{thm:CLT}.
\end{pfof}

\section{Decay of correlations for surfaces with flat cylinder}
\label{sec:decay}

In this section, we consider decay of correlations for the geodesic flow $g_t:M\to M$ and the global Poincar\'e map $g:\Sigma\to\Sigma$.
The general setup is the same as in~\cite[Section~7]{LMM24}.
We recall the basics here.\footnote{We caution that for example $\tau$ here was called $\sigma$ in~\cite{LMM24}.}
\begin{enumerate}[$\bullet$]
\item The flow $g_t$ and map $g$ are related by the identity
$g=g_h$ where $h:\Sigma\to\R^+$ is a H\"older roof function with $0<\inf h<\sup h<\infty$.
\item 
The first return maps $g:\Sigma\to\Sigma$ and $f:\Sigma_0\to\Sigma_0$ are related by $f=g^R$ where $R:\Sigma_0\to\Z^+$ is a discrete first return time.
\item 
The uniformly hyperbolic first return map $f:\Sigma_0\to\Sigma_0$ is modelled by a Young tower with exponential tails~\cite{Young98}: there is a subset $Y\subset\Sigma_0$ with product structure and bounded distortion, a probability measure $\mu_Y$ on $Y$, and a ``good'' return time $\tau:Y\to\Z^+$ constant along stable leaves with $\mu_Y(\tau>n)=O(e^{-cn})$ for some $c>0$, such that $f^\tau:Y\to Y$ is measure-preserving. Moreover, quotienting along stable leaves yields a full-branch Gibbs-Markov map $F:Z\to Z$ as in Section~\ref{sec:CLT}, and $\tau$ is constant on the associated partition elements.
\item 
The global Poincar\'e map $g:\Sigma\to\Sigma$ is modelled by a subexponential Young tower over $g^{\varphi_*}=(g^R)^\tau=f^\tau:Y\to Y$ where the return time $\varphi_*:Y\to\Z^+$ is given by $\varphi_*=\sum_{\ell=0}^{\tau-1}R\circ f^\ell$.
\item 
We can regard $f^\tau:Y\to Y$ as a (non-first return) Poincar\'e map for the flow $g_t$ with roof function $\varphi=h_{\varphi_*}=\sum_{\ell=0}^{\varphi_*-1}h\circ g^\ell$.
\end{enumerate}

By Lemma~\ref{lem:RC}(1), we have precise asymptotics for the tails of 
$R_\cC:\Sigma_0\to\Z^+$. In particular,
$\mu_{\Sigma_0}(R_\cC>n)\sim \sigma_R^2 n^{-2}$.
Usually, it is difficult to leverage this into precise asymptotics for $\varphi_*:Y\to\Z^+$.
However, here we can exploit the fact that $R$ satisfies a nonstandard limit law by following ideas in~\cite{Gouezel10,MV20}.

\begin{prop} \label{prop:tail}
$\mu_Y(\varphi_*>n)\sim\bar\tau\sigma_R^2 n^{-2}$.
\end{prop}

\begin{proof} 
By Corollary~\ref{cor:RCLT}, $R$ 
satisfies a nonstandard CLT
with variance $\sigma_R^2>0$.
Now apply~\cite[Lemma~4.1]{LMMsub}.
\end{proof}

We can now state and prove our result on upper and lower bounds on decay of correlations for the global Poincar\'e map $g$ as advertised in Remark~\ref{rmk:lower}.

\begin{thm} \label{thm:decaymap}
For all H\"older observables $v,w:\Sigma\to\R$, there is a constant
$C>0$ such that 
$$
\left|\int_{\Sigma} v \cdot (w\circ g^n)\,d\mu_{\Sigma}-\int_{\Sigma}v\,d\mu_{\Sigma}\int_{\Sigma}w\,d\mu_{\Sigma}\right|
\leq C n^{-1}
\quad\text{for all $n\ge1$}.
$$
Moreover, for all H\"older observables $v,w:\Sigma\to\R$  supported in $\Sigma_0$,
\[
\int_{\Sigma} v \cdot (w\circ g^n)\,d\mu_{\Sigma}-\int_{\Sigma}v\,d\mu_{\Sigma}\int_{\Sigma}w\,d\mu_{\Sigma}\sim \bar\tau\sigma_R^2
n^{-1} \int_\Sigma v\,d\mu_\Sigma\, \int_\Sigma w\,d\mu_\Sigma \quad\text{for all $n\ge1$}.
\]
\end{thm}

\begin{proof} 
Recall that the global Poincar\'e map $g:\Sigma\to\Sigma$ is modelled by a (subexponential) two-sided Young tower $\Gamma$ with return time $\varphi_*:Y\to\Z^+$.
By Proposition~\ref{prop:tail}, $\mu_Y(\varphi_*>n)\sim \bar\tau\sigma_R^2n^{-2}$. 
Hence the upper bound follows from~\cite{Young99} (a complete argument can be found in~\cite[Theorem~2.10]{KKM19} or~\cite[Appendix~B]{MT14}) and the lower bound from~\cite[Theorem~7.4]{BMT21}.
\end{proof}

\begin{lemma} \label{lem:varphi}
There is a constant $C>0$ such that
$|\varphi(y)-\varphi(y')|\le C d(y,y')$ for $y,y'\in Y$ with $y'\in W^s_y$, and
$|\varphi(y)-\varphi(y')|\le C d(f^\tau y,f^\tau y')$ for $y,y'\in Y$ with $y'\in W^u_y$.
\end{lemma}

\begin{proof} This follows from Lemma~\ref{lem:Ups} in the same way 
that~\cite[Lemma~7.1]{LMM24} follows from~\cite[Lemma~5.11]{LMM24}.
\end{proof}

\begin{pfof}{Theorem~\ref{thm:decay}}
In the terminology of~\cite[Section~6]{BBM19}, we have shown that $g_t$ is a ``Gibbs-Markov flow''. Their condition~(H) follows from Lemma~\ref{lem:varphi} by~\cite[Lemma~8.3]{BBM19}.
As explained in~\cite[Section~7]{LMM24}, the ``absence of approximate eigenfunctions'' condition in~\cite{BBM19} is automatic due to the fact that the geodesic flow $g_t$ has a contact structure.
Since $h$ is bounded above and below, it follows from Proposition~\ref{prop:tail}  that $\mu_Y(\varphi>t)\approx t^{-2}$.
Hence we can apply~\cite[Theorem~6.4]{BBM19}.
\end{pfof}

\appendix

\section{Chernov axioms}
\label{app:C}

For convenience, we list the Chernov axioms as stated in~\cite[Section 3]{LMM24}. A more in-depth discussion of the axioms, and where they differ from treatments in~\cite{Balint-Toth,Chernov-1999,Chernov-Zhang}, is given in~\cite{LMM24} and not repeated here.

\medskip
\noindent
{\bf (A1)} Let $(\Sigma_0,d)$ be a compact $C^2$ Riemannian surface.
We let $\mfS^+,\mfS^-$
be closed subsets of $\Sigma_0$, and let
$f:\Int \Sigma_0\setminus \mfS^+\to \Int\Sigma_0\setminus\mfS^-$ be a $C^2$ diffeomorphism.

\medskip
For $n\ge 1$, define
$\mfS_n=\mfS^+\cup f^{-1}\mfS^+\cup\cdots\cup f^{-n+1}\mfS^+$.

\medskip
\noindent
{\bf (A2)} There are two families of cones $\{C^u_x\},\{C^s_x\}$ in
the tangent planes $T_x\Sigma_0$, $x\in \Sigma_0$, called unstable and stable cones,
and a constant $\Lambda>1$ such that:
\begin{enumerate}[$\bullet$]

\parskip=-2pt
\item $\{C^u_x\},\{C^s_x\}$ are continuous on $\Sigma_0$.
\item The axes of $\{C^u_x\},\{C^s_x\}$ are one-dimensional.
\item $\min\limits_{x\in \Sigma_0}\angle (C^u_x,C^s_x)>0$.
\item $Df(C^u_x)\subset C^u_{fx}$ and $Df(C^s_x)\supset C^s_{fx}$
whenever $Df$ exists.
\item 
$\|Df(v^u)\|\geq\Lambda\|v^u\|$ for all $v^u\in C^u_x$, and 
$\|Df^{-1}(v^s)\|\geq\Lambda\|v^s\|$ for all $v^s\in C^s_x$.
\end{enumerate}

Given a curve $W\subset \Sigma_0$, let $d_W$ be the Riemannian metric on $W$ induced by $d$.
Similarly, let $m_W$ be Lebesgue measure on $W$ induced by Lebesgue measure $m$ on ${\Sigma_0}$.

Recall the definition of LUM/LSM in Subsection~\ref{ss:hyperbolicity-f}.
Let $W_1,W_2$ be sufficiently
small and close enough LUMs such that small LSMs intersect each of $W_1,W_2$ at most once.
Let $W'_1=\{x\in W_1:W^s_x\cap W_2\neq\emptyset\}$ and $H:W'_1\to W_2$ be the \emph{holonomy map},
 i.e.\ $H(x)$ is the unique intersection between $W^s_x$ and $W_2$.
Also, let $\Lambda(x)=|\det (Df|_{T_xW^u_x})|$.

\medskip
\noindent
{\bf (A3)} 
The angles between $\mfS^+$ and LUMs, 
and between $\mfS^-$ and LSMs, are bounded
away from zero.

\medskip
\noindent
{\bf (A4)} 
The map $f$ preserves an ergodic volume measure $\mu_{\Sigma_0}$
such that a.e.\ $x\in \Sigma_0$ has an LUM $W^u_x$ and the conditional measure on $W^u_x$
induced by $\mu_{\Sigma_0}$ is absolutely continuous with respect to $m_{W^u_x}$. Furthermore,
$f^n$ is ergodic for all $n\geq 1$.

\medskip
\noindent
{\bf (A5)} The leaves $W^{u/s}_x$ are uniformly $C^{1+{\rm Lip}}$.

\medskip
\noindent
{\bf (A6)} 
There is a function $\psi:[0,\infty)\to[0,\infty)$ with $\lim_{x\to0}\psi(x)=0$
for which the following holds: if $W$ is an LUM, then for all $x,y$ belonging to the same connected component $V$
of $W\cap \mfS_{n-1}$,
$$
\log\left[\prod_{i=0}^{n-1}\dfrac{\Lambda(f^i x)}{\Lambda(f^i y)}\right]\leq \psi\big(d_{f^nV}(f^n x,f^n y)\big).
$$

\noindent
{\bf (A7)} There is a constant $C>0$ with the following property:
if $W_1,W_2$ are two sufficiently small and close enough LUMs,
then the holonomy map $H:W_1'\to W_2$ is absolutely continuous with respect to $m_{W_1},m_{W_2}$
and
$$
C^{-1}\leq \frac{m_{W_2}(H[W_1'])}{m_{W_1}(W_1')} \leq C.
$$

\noindent
{\bf (A8)}
Write $\mfS_1=\mfS^+$ as a disjoint union $\mfS_1=\mfS^+_{\rm P}\cup \mfS^+_{\rm S}$,
and let 
\[
\mfS^+_{{\rm P},n}=\mfS^+_{\rm P}\cup f^{-1}\mfS^+_{\rm P}\cup\cdots\cup f^{-n+1}\mfS^+_{\rm P}.
\]
We assume:
\begin{enumerate}[$\bullet$]

\parskip=-2pt
\item 
For each LUM $W$, the intersection $W\cap \mfS_1$ is at most countable and
$W\cap\mfS_1$ has at most one accumulation point $x_\infty$.
Moreover, there are constants $C,d>0$
such that if $x_n$ is the monotone sequence in $W\cap\mfS_1$ converging to $x_\infty$ then
$d(x_n,x_\infty)\leq Cn^{-d}$ for all $n\geq 1$.
\item 
Let $\Lambda_n=\min\{\Lambda(x):x\in W_n\}$ where
$\{W_n\}$ are the connected components of $W\setminus \mfS^+_{\rm S}$; then
$$
\theta_0=\liminf_{\delta\to 0}\sup_{|W|<\delta}\sum_n \Lambda_n^{-1}<1,
$$
where the supremum is taken over LUMs $W$ with $W\cap\mfS^+_{\rm P}=\emptyset$.
\item Let
$
K_n=\lim_{\delta\to 0}\sup_{|W|<\delta}K_n(W)
$
where the supremum is taken over LUMs $W$ and $K_n(W)$ is the number
of connected components of $W\setminus \mfS^+_{{\rm P},n}$; then
$K_n<\min\{\theta_0^{-1},\Lambda\}^n$ for some $n\geq 1$.
\end{enumerate}

\paragraph{Acknowledgements}
 YL was supported by CNPq/MCTI/FNDCT project 406750/2021-1,
FUNCAP grant UNI-0210-00288.01.00/23, Instituto Serrapilheira grant
``Jangada Din\^{a}mica: Impulsionando Sistemas Din\^{a}micos na Regi\~{a}o Nordeste'', and
FAPESP grant number 2025/11400-7.
The research of IM was supported in part by FAPESP grant number 2024/22093-5.
The authors are grateful for the hospitality of Universidade Federal do Ceará (UFC), where
much of this work was carried out.

\end{document}